\documentclass[12pt,reqno,a4paper]{amsart}
\usepackage{euscript,fullpage,amsmath,amssymb}
\usepackage[T1]{fontenc}
\usepackage{enumerate}
\usepackage{tikz}
\usepackage{graphicx}
\usepackage{float}
\usepackage{setspace}
\usepackage{amssymb}
\usepackage{subcaption}
\usepackage[utf8]{inputenc}

\newtheorem{proposition}{Proposition}

\newtheorem{remark}{Remark}

\title{On the computation of the extremal index for time series}
 \date{\today}

\begin{document}
 \begin{abstract}
The extremal index is a quantity introduced in extreme value theory to measure the presence of clusters of exceedances. In the dynamical systems framework, it provides important information about the dynamics of the underlying systems. In this paper we provide a review of the meaning of the extremal index in dynamical systems.  Depending on the observables used, this quantity can inform on local properties of attractors such as periodicity, stability and persistence in phase space, or on global properties such as the Lyapunov exponents. We also introduce a new estimator of the extremal index and shows its relation with those previously introduced in the statistical literature. We reserve a particular focus to the systems perturbed with noise as they are a good paradigm of many natural phenomena. Different kind of noises are investigated in the annealed and quenched situations. Applications to climate data are also presented.
\end{abstract}
 \maketitle
\begin{center}
\authors{Th. Caby \footnote{Aix Marseille Universit\'e, Universit\'e de Toulon, CNRS, CPT, 13009 Marseille, France and Center for Nonlinear and Complex Systems, Dipartimento di Scienza ed Alta Tecnologia,
Universit\`a degli Studi dell'Insubria, Como, Italy. E-mail: {\tt \email{caby.theo@gmail.com} };},
D. Faranda\footnote{Laboratoire des Sciences du Climat et de l'Environnement, UMR 8212 CEA-CNRS-UVSQ,
IPSL and Universit\'e Paris-Saclay, 91191 Gif-sur-Yvette, France and
London Mathematical Laboratory, 8 Margravine Gardens, London, W6 8RH, UK. Email: {\tt \email{davide.faranda@lsce.ipsl.fr}}.},
S.\ Vaienti\footnote{Aix Marseille Universit\'e, Universit\'e de Toulon, CNRS, CPT, 13009 Marseille, France. E-mail: {\tt \email{vaienti@cpt.univ-mrs.fr}}.}},
P. Yiou\footnote{Laboratoire des Sciences du Climat et de l'Environnement, UMR 8212 CEA-CNRS-UVSQ,
IPSL and Universit\'e Paris-Saclay, 91191 Gif-sur-Yvette, France. E-mail: {\tt \email{pascal.yiou@lsce.ipsl.fr}}.}

\end{center}
\tableofcontents
\section {Introduction}
In the last ten years extreme value theory (EVT) has been successfully applied to the study of time series generated by dynamical systems. Illustrations can be found in the book \cite{book} and the articles \cite{luca1,luca2,luca3} for an exhaustive account on the formalism, the methodologies and several applications. In particular the extremal index (EI) --- a quantity defined in the unit interval --- has been used as a powerful statistical indicator to discriminate among different qualitative types of dynamical behaviors in the phase space of a few climate models and in real situations (an extensive overview can be found in \cite{davideei}). In a series of papers \cite{nature, messori, FF, davideei,nature2},  the EI was renamed as {\em local persistence} indicator, suitable to estimate the average cluster size of the trajectories within the neighborhood of a given point. The aim of this note is to give an overview on a few recent rigorous mathematical results which show that the EI is very sensitive to  stochastic perturbations affecting the deterministic evolution of a system states. We will also discuss the influence of noise on the {\em dynamical extremal index}, which was recently introduced to characterize the local divergence of chaotic orbits\cite{D2}.

 The first rigorous computations of the EI for  expanding and uniformly hyperbolic systems (from now on named {\em chaotic systems}), showed a dichotomy for the value of the EI, which is equal to $1$ everywhere, except on periodic points.  \cite{JF} give a detailed account of this matter. On the other hand, the EI exhibits local variations which have been also connected to the local fractal dimensions. By assuming that the  data we  are interested in are chaotic, one could wonder about the   wider  variation of the extremal index, offering a spectrum of values beyond the previous dichotomy. In a recent paper \cite{CFMVY} we commented about a similar effect concerning the computation of the local  dimensions. Indeed, while it is well known for a large
class of systems that the local dimensions are all equal to the so-called  information dimension $D_1$ with probability one, large deviation theory estimates
the likelihood of deviations from this value at finite resolution. In this perspective, the spread of the
experimentally observed values of the local dimensions  can be thought of as originating from the
multifractal structure of the invariant measure, which in turn is revealed
by the non-constant value of generalized dimensions.

Although finite resolution affects the  estimation of the extremal index, it is unlikely that it  enjoys some large deviation property.  This is due to the previous dichotomy, which prevents the existence of a smooth spectrum of values for the EI.  Yet, the computation of the  EI is very sensitive to   randomness and disturbances in the measuring process.  The main object of our note is to show how the EI is affected by the presence of noise. Our starting point is to assume for the EI a general formula obtained by Keller and Liverani \cite{KL} in the context of stationary  dynamical systems enjoying the so-called {\em spectral gap property}. For this reason we call it the {\em spectral formula}. It encompasses all the other rigorous formulae previously obtained for the computation of the extremal index. We will show that when the spectral formula holds, one can obtain a generalized version of the O'Brien formula,  which inspired several numerical algorithms and statistical estimators (see Eq. (3.2.4) in \cite{book}). This is discussed in Proposition \ref{propo} and Remark \ref{IIRR}, which constitute two main contributions of this paper.

The first sections are devoted to the computation of the EI for the
standard observable given by the distribution of the first visit of the trajectory in a decreasing
net of balls around a given point, which we qualify as {\em the target set}.  The system will be perturbed
according to three different classes of noises: sequential, quenched and annealed.  In particular, we will 
focus on the annealed class since it gives  stationary random processes particularly suitable for the application
of the spectral formula. We will show on simple but non trivial  examples that noises with discrete distributions tend to make
the EI less than one, while noises with continuous distributions usually make the extremal index equal to one.
The latter case applies also to deterministic dynamics with the target set affected by disturbances
with smooth densities, which corresponds to common physical situations.

 The final section deals with a different extremal index
we introduced recently with the name of Dynamical Extremal Index and  that
explores the mutual distances between the coordinates of a point moving in a suitable higher
dimensional direct product space. This new index captures the rate of phase space
contraction and  is naturally related to positive Lyapunov exponent(s). We will show again how this index
is sensitive to the different kinds of perturbations described above.

We finally analyze how the random perturbations affect the statistics of the number of visits in the target regions. The expectation
of the limiting law is the reciprocal of the extremal index and the differences in the value of the EI  reflect
in different types of Poisson compound distributions.

\section{The deterministic case}\label{DET}
Our intent is to provide a critical discussion of the application of the EI to time series. Therefore, we start defining it for particular random processes. Let us therefore consider a discrete dynamical systems defined by a map $T$ acting on a smooth manifold $X$ and preserving a Borel measure $\mu$. Suppose $z$ is a point of $X$ and $B(z,r)$ an open ball around $z$ and of radius $r$. Given any other point $x\in X$ let us consider the random variables $H_n(x)$ given by
 \begin{equation}\label{H}
 H_n(x):=\text{\{first time the iterate $T^n(x)$ enters the open ball $B(z, e^{-u_n})$\}},
 \end{equation}
 where the {\em boundary level} $u_n$ is defined by asking that
 \begin{equation}\label{U}
n\  \mu(B(z, e^{-u_n}))\rightarrow \tau,
 \end{equation}
for some positive number $\tau$. Notice that by the stationarity (or invariance) of the measure $\mu$, the quantity $\mu(B(z, e^{-u_n}))$ gives the probability that any iterate of the map $T$ be in the open balls $B(z, e^{-u_n})$.  Moreover as soon as the measure is not atomic, the measure of a ball varies continuously with the radius and this allows us to explicit $u_n$ as a function of $\tau$ and $n$.

We say that we have an extreme value law for the process $H_n$ if
\begin{equation}\label{EVT}
\mu(x; H_n(x) > u_n)\rightarrow e^{-\theta \tau},
\end{equation}
where $0\le \theta\le 1$ is called the {\em extremal index}.

 We remind that the events $H_n > u_n$ are equivalent to the following ones:
\begin{equation}\label{MM}
\{M_n(x)\le u_n\} , \ \text{where} \ M_n=\max\{\phi(x), \phi(Tx), \dots, \phi(T^{n-1}x\},
\end{equation}
and the observable $\phi$ is defined as
\begin{equation}\label{OO}
\phi(x)=-\log d(x,z).
\end{equation}

It is remarkable that for a large class of systems whose transfer operator admits a spectral gap, the extremal index can be explicitly computed. This is done    by G. Keller \cite{GK}, who applied  the perturbative theory developed in \cite{KL}.  This spectral approach has the advantage to give a formula for the EI which holds for general target sets when their measure goes to zero. We will use it in the last section of this note when balls are replaced by tubular neighborhoods.

  Let us first define the event
\begin{equation}\label{FFFF}
\Omega_n^{(k)}(z):= \{\phi(x) > u_n,  \max_{i=1,...,k} \phi(T^ix)\le u_n, \ \phi(T^{k+1}x)> u_n\},
\end{equation}
namely the set of points in $B(z, e^{-u_n})$ whose first $k$ iterates are outside $B(z, e^{-u_n})$ and whose $(k+1)$-th iterate falls again in $B(z, e^{-u_n})$.

Suppose that the following limit exists:
\begin{equation}\label{QK}
q_k:=\lim_{n\rightarrow \infty} \frac{\mu(\Omega_n^{(k)}(z))}{ \mu(B(z, e^{-u_n}))}.
\end{equation}
Then the perturbative theory gives an estimate of the EI:
\begin{equation}\label{TH}
\theta=1-\sum_{k=0}^{\infty}q_k.
\end{equation}
This formula allows us to reproduce a few existing results whenever the target point $z$ is of a particular type. For instance, for one-dimensional expanding systems with strong mixing properties and preserving a measure absolutely continuous with respect to Lebesgue, the extremal index is $1$ everywhere but in periodic points. When we have a periodic point  $z$  of (minimal) period $p,$ only the $q_{p-1}$ term is nonzero and it is of the form \cite{GK, JF}
\begin{equation}
q_{p-1}=\frac{1}{|DT^p(z)|}.
\label{DT}
\end{equation}
The dependence of the EI on the periodicity of the target point $z$ means that the sojourns and returns of the point into the ball $B(z, e^{-u_n})$, which constitute the {\em clusters of exceedances}, keep memory either of the past orbit and of the topological structure of the target point.

This can be made more precise by defining a cluster size distribution $\pi_n(j)$ as the probability of having $j$ returns into the ball up to a rescaled time $n/k_n$, for a suitable  sequence $k_n=o(n)$.   Section 3.2.1 in \cite{book} provides  more details on this derivation; a more formal definition is also given in section 7. It can be therefore proved that (see Sec. 3.3.3 in \cite{book}):
\begin{equation}\label{TH2}
\theta^{-1}= \lim_{n\rightarrow \infty}\sum_{j=1}^{\infty}j\pi_n(j),
\end{equation}
which is interpreted  by saying that the {\em EI is equal to the inverse of the average
cluster size}.\footnote{Abadi et al.  \cite{AAFFFF} built  a  dynamically generated stochastic processes with an extremal index for which that equality does not hold. They considered observable functions maximised at  least two
points of the phase space, where one of them is an indifferent periodic point and another
one is either a repelling periodic point or a non periodic point. We will not consider these kind of observables in this paper.}
In the applications  to dynamical systems, the periodicity
translates into the fact that the clustering $\pi=\lim_{n\rightarrow \infty}\pi_n$ is actually a geometric distribution of parameter $\theta \in (0,1]$, i.e., $\pi_k=\theta(1-\theta)^{k-1}$, for every $k\in \mathbb{N}_0$ \cite{HV}. We will return to this matter in section 7.

The interpretation of the EI given in Eq. (\ref{TH}) has been used in applications to time series climate data and it was also emphasized the {\em local} character of such an indicator and its strong correlation with the fractal local dimension of the invariant measure \cite{nature, messori, FF}.

Let us now consider for $ k\in \mathbb{N}_0$ the event (compare with Eq. (\ref{FFFF})):
\begin{equation}\label{EEEE}
    A^{(k)}_n:=\{\phi(x) > u_n \cap \max_{i=1,\dots,k} \phi(T^ix) \le u_n\}
     \end{equation}
and the quantity:
\begin{equation}
\theta_k^{(n)}:=\frac{\mu(
    A^{(k)}_n)}{\mu(\phi(x) > u_n)}
\end{equation}
where we set $\theta_0^{(n)}=1, \forall n$. By introducing the event $\phi(T^{k+1}\le u_n)$ and its complement, and passing to the limit for $n\rightarrow \infty$, we  get
\begin{equation}\label{QQTT}
q_k=\theta_k-\theta_{k+1},
\end{equation}
where
\begin{equation}\label{lil}
\theta_k=\lim_{n\rightarrow \infty}\theta_k^{(n)},
\end{equation}
when the limit exists.

Notice that as soon as one of the two sequences
$\{q_k\}_{k\in \mathbb{N}}, \{\theta_k\}_{k\in \mathbb{N}}$
is known, the other one is determined as well.

In some circumstances, the limit in Eq. (\ref{lil}) gives exactly the extremal index.
To explain  this point,  we   need to be more precise about the assumptions we make. The spectral approach briefly sketched above applies to a large class of dynamical systems admitting a spectral gap for the transfer (Perron-Fr\"obenius) operator. This  means that the systems have exponential decay of correlations for smooth enough  observables, usually of bounded variation type. This is not sufficient to establish the existence of the limit in Eq. (\ref{QK}). Computations of that limit under various circumstances are  given  in \cite{GK, KL, AFV, CFMVY}. The more standard approach to EVT, which  consists in adapting the classical Leadbetter  theory for i.i.d. random variables to stationary dependent processes \cite{LEAD}, requires  milder mixing conditions that allow to get asymptotic independence for the process.
The theory of O'Brien \cite{OB} was  developed for strictly stationary processes verifying asymptotic independence, which is a strong probabilistic  mixing condition like  the $\phi$-mixing (see \cite{bradely} for the definition). We  stress that asymptotic independence is very difficult to check in practice.  Recently it has been shown that weaker conditions are enough. For our purposes,  it is sufficient to remind condition $D(u_n)$ in section 2.3 of \cite{book} successively improved in \cite{JF}\footnote{This condition holds for instance when the invariant measure $\mu$ is mixing with  decay of correlation fast enough; sometimes a rate of decay as $n^{-2}$ is sufficient.} Asymptotic independence is not enough to get the convergence of the limit distribution of the maxima. One needs to control short returns, which could be in particular affected by clustering. A powerful condition which takes care of that is:
\begin{equation}\label{DD}
D^{(k)}_{u_n}=\lim_{n\rightarrow \infty}n \ \sum_{j=k+1}^{[n/ l_n]-1}\mu(A^{(k)}_n \cap T^{-j}(A^{(k)}_n))=0,
\end{equation}
where $l_n=o(n)$ is a sequence slowly diverging to infinity and verifying the $D(u_n)$ conditions (see Remark 4.1.2 in \cite{book} for the details). We call the integer $k$ the {\em clustering order}. We  now summarize in the following proposition a few  facts that are useful for our paper and which refer to the stationary process $\{\phi \circ T^k\}_{k\in \mathbb{N}}$,
on the probability space $(X, \mu)$:
\begin{proposition}\label{PROPP}
\begin{itemize}
\item Suppose the sequence $u_n$ verifies Eq. (\ref{U}) and condition $D(u_n)$ holds. Moreover suppose that $\liminf_{n\rightarrow \infty}\mu(M_n\le u_n)>0$ and condition in Eq. (\ref{DD}) holds with clustering order $k$. Then
    $$
    \mu(M_n\le u_n)- e^{-\tau \theta^{(n)}_k}\rightarrow 0,
    $$
    as shown by  \cite{CH}, Proposition 1.1.
    \item If the limit in Eq. (\ref{lil}) exists, then {\em the extremal index is given by $\theta_k$}, as shown by Corollary 1.3 in \cite{CH} or  Corollary 4.1.7 in \cite{book}. This is sometimes called the O'Brien formula.
        \item If condition Eq. (\ref{DD}) holds for some particular $k'\in \mathbb{N}$, it also holds  for all $k\ge k'$ and therefore the limit in Eq. (\ref{lil}) for all $k\ge k'$  gives again the extremal index $\theta$, as shown by \cite{JF} and \cite{book}, Remark 4.1.10.
\end{itemize}
\end{proposition}

        When the second item in the Proposition holds  for a process with clustering order $k$ we have
\begin{equation}\label{CO1}
q_j=\theta_j-\theta_{j+1}=0, \ \forall j\ge k.
\end{equation}
The natural question is therefore to ask when condition (\ref{DD}) is verified. For a large class of dynamical systems, in particular verifying the assumptions of theorem 4.2.7 in \cite{book},  it can be shown that condition (\ref{DD}) holds with $k=0$ if the target point $z$ is not periodic, and with $k=p$ if $z$ is periodic of prime period $p$ (see Proposition 4.2.13 in \cite{book}). In particular,    for one-dimensional expanding systems with strongly mixing properties and preserving a measure absolutely continuous with respect to Lebesgue, the extremal index is
$1-q_{p-1}=1-\frac{1}{|DT^p(z)|}$
in periodic points $z$  of (minimal) period $p$. For higher dimensional systems, condition (\ref{DD}) gives a precise value for $\theta$ (see for instance \cite{diffeo}).
\begin{remark}\label{CR}
  For the preceding example around periodic points, the spectral and the standard approaches both give $q_{p-1}$ and $\theta_p$ different from zero. Yet, the spectral approach shows that all $q_j$ are zero for $j\neq p-1$. This is therefore coherent with Eq. (\ref{CO1}) for which  $\theta_j=\theta_{j+1}$ for  $j\ge p$.    Let us now consider $q_{p-1}=\theta_{p-1}-\theta_p$. The quantity $\theta_{p-1}$ is well defined as the limit of $\theta_{p-1}^{(n)},$ when $n$ goes to infinity and it is equal to $1$. This implies that the condition (\ref{DD}) is violated, because otherwise $\theta_p$ should be equal to $1$ as well.  The fact that all the $\theta_{j}$, for $ j\le p-1$ are equal to $1$ is not surprising: it means that the full conditional  measure of points in the ball $B(z, e^{-u_n})$ are outside it when iterated up to $p-1$ times.
\end{remark}
In the assumptions of Proposition \ref{PROPP}, the EI is equal to $\theta_k$ when the limit in Eq. (\ref{lil}) defining this quantity exists and $k$ being the clustering order. We now show that this is a particular case of a more general result.
\begin{proposition}\label{propo}
Let us suppose that the limits in Eq. (\ref{QK}) defining the quantities $q_k$ exist for any $k\ge 0$. Then the extremal index $\theta$ is given by
\begin{equation}\label{FF}
\theta=\lim_{k\rightarrow \infty}\theta_k.
\end{equation}
\end{proposition}
\begin{proof}
We noticed that whenever the sequence $q_k$ exists, the same happens for the sequence $\theta_k$. Then by a simple telescopic trick we get
 $$
\theta=\lim_{k\rightarrow \infty}\left(1-\sum_{j=0}^kq_j\right)=\lim_{k\rightarrow \infty}\theta_{k+1}.
$$
\end{proof}
\begin{remark}\label{IIRR}
Under the assumptions of the preceding proposition, we got that
\begin{equation}\label{GOF}
\theta=\lim_{k\rightarrow \infty}\lim_{n\rightarrow \infty}\theta_{k}^{(n)}.
\end{equation}

We stress that the limits in Eq. (\ref{lil}) or Eq. (\ref{QK}) have been computed up to now and in the framework of dynamical systems only around periodic points, or around the diagonal in coupled systems (see \cite{cml}) where at most only one of them is different from zero. We will present later on  further examples of explicit computations of the $q_k$. In particular section \ref{DN} will give an example of a stationary random dynamical system for which all the $q_k$ are different from zero and from each other, which implies that condition (\ref{DD}) is violated at all orders. In this respect our Eq. (\ref{GOF}) is a generalization of the O'Brien formula.
\end{remark}

\subsection{Statistical estimators of the extremal index}
\label{est}
We saw in the preceding section that  the value $\theta_p$ provides the extremal index if the target point $z$ is of (minimal) period $p$. It is therefore not surprising that
estimators based on the O'Brien formula have been proposed for the computation of the EI. In particular the following formula has been introduced to compute the extremal index associated to time series in \cite{wiley} (formula 10.26): we start with
 a trajectory $(x_0,Tx_0,\dots,T^{N-1}x_0)$ and for a high enough threshold $u$ we compute:

\begin{equation}\label{MC}
\hat\theta_m :=\frac{ \sum_{i=0}^{N-1-m} \mathbf{1}(\phi(T^i x_0) > u \cap \max_{j=1,...,m} \phi(T^{j+i}x_0) \le u)/(N-m)}{ \sum_{i=0}^{N-1}\mathbf{1}(\phi(T^ix_0) \ge u)/N}.
\end{equation}

If we are in presence of a periodic point of period $m$ for the observable $\phi(x)=-\log d(z,x)$, then $\hat\theta_m$ is a good approximation of the EI.

We notice that when $m=1$,   Eq. (\ref{MC}) gives the
likelihood estimator of S\"uveges $\hat\theta_{Su}$ \cite{suveges}, and the obtained result is an indicator of the persistence of the system close to a state $z$ (see for example \cite{nature}). We remind that for a high quantile $p$ of the time series, the S\"uveges estimator  is given by:

 \begin{equation}\label{suveges}
\hat\theta_{Su}=\frac{\sum_{i=1}^{N_c}(1-p)S_i +N+N_c -\sqrt{(\sum_{i=1}^{N_c}(1-p)S_i +N+N_c)^2-8N_c\sum_{i=1}^{N_c}(1-p)S_i}}{2\sum_{i=1}^{N_c}(1-p)S_i},
\end{equation}

where $N$ is the number of exceedances over $\tau$, $N_c$ is the number of clusters of two or more exceedances, and $S_i$ is the size of the cluster $i$. This estimator is presented in \cite{suveges}. A {\em Matlab} code and numerical details related to this computation in the context of dynamical systems are available in \cite{book}.

We warn that in presence of periodicity of order, say $p$, the estimator of Eq. (\ref{MC}) is reliable when it is strictly less than $1$, because all the $\theta_m$, for $m<p$ are equal to $1$. This observation applies as well to the computation of the $q_k$: they are all zero up to $q_{p-1}$.

Ferreira \cite{ferreira} reviews several  numerical algorithms  to compute the EI (including those shown before). In this paper we choose to work with the $q_k$ and we list a few features of this approach.

\begin{itemize}
\item In the presence of periodicity,  we do not need to check condition Eq. (\ref{DD});
 the clustering order, if any,   is obtained automatically by progressing in the evaluation of the $q_j$.  We will show in sections \ref{DN} and \ref{MMTT} examples of stationary processes constructed with random perturbation where several $q_k$ are different from zero.
    \item Eq. (\ref{TH}) works as well when the ball $B(z, e^{-u_n})$ is replaced by any other measurable sets whose measure goes to zero when $n\rightarrow \infty$. This comes to modify suitably the observable $\phi$. A concrete example will be given in section \ref{SDEI}.
        \item Whenever one of the $q_j$ is strictly positive, we get an EI strictly less than $1$.  This will be extensively used in sections \ref{DN}, \ref{MMTT}, \ref{DND}. Notice that it is equivalent to find a certain   $\theta_k<1$ for $k\ge 1$.

\end{itemize}

 In practice, one cannot compute all the $q_j$ values and must stop at an order $m$. We stress that the sequence $\{q_j\}$ is rapidly  decreasing to $0$ and the quantity $\theta_m=1-q_0- \cdots -q_{m-1}$, where $m$ should be taken not too big, should give a fair estimate of $\theta$.  The approximate value $\hat{q_j}$ for $q_j$ is obtained by considering a trajectory of length $N$ as

\begin{equation}\label{AAQ}
\hat q_j=\frac{ \sum_{i=0}^{N-2-m} \mathbf{1}(\phi(T^i x_0) \ge u \cap \max_{l=1,\dots,m} \phi(T^{l+i}x_0) < u \cap \phi(T^{i+j+1}x_0) \ge u)/(N-1-m)}{ \sum_{i=0}^{N-1}\mathbf{1}(\phi(T^ix_0) \ge u)/N}.
\end{equation}

\subsection{Numerical estimates of the extremal index in deterministic systems}

We present below the computations for different maps and various target points $z$ and show the difference between the approximate estimate $\hat\theta_m$ of $\theta_m$ (we performed our computations at the order $m=5$) and S\"uveges estimate (with order $0$). We find that $\hat\theta_m$ obtains values close to the predictions of Eq. (\ref{DT}), while the other estimate gives $1$ for points of period larger than $1$.

\begin{tabular}{|c|c|c|c|c|c|c|}
  \hline
  Application &  $z$ & period & Theoretical value & $\hat\theta_{Su}$ & $\hat\theta_5$ & Uncertainty\\
  \hline
  $2x$ mod1 & 4/5 & 4 & 0.9375 & 1 & 0.9375 & 0.0024\\
  $2x$ mod1 & 0 & 1 & 0.5 & 0.5007 & 0.5005 & 0.0028 \\
  $2x$ mod1 & 1/3 & 2 & 0.75 & 1 & 0.7508 & 0.0017\\
  $2x$ mod1 & $1/\pi$ & not periodic & 1 & 1& 1 &0\\
  Cat's map & $(1/3,2/3)$ & 4 & 0.9730 \cite{diffeo}& 1  & 0.9732 & $8.3.10^{-4}$ \\
  Cat's map & $(1/2,1/2)$ &3 & 0.9291 \cite{diffeo} & 1 & 0.9292 & $9.32.10^{-4}$ \\
  Cat's map & $(0,0)$ & 1 & 0.5354 \cite{diffeo} & 0.5352 & 0.5350 & $0.0018$ \\
  Cat's map & $(1/\sqrt2,\pi-3)$ & not periodic & 1 & 1 &1 &0\\
  \hline
\end{tabular}\\

\captionof{table}{Comparison of estimates of $\theta$ found with the different methods for the $2x-\mathrm{mod} 1$ map and Arnol'd cat map \cite{cat}, defined in $\mathbb{T}^2$ by $T(x,y)=(x+y,x+2y)$. For all of the computations, we averaged our results over $20$ trajectories of $5.10^7$ points and took as a threshold the $0.999$-quantile of the observable distribution. The uncertainty is the standard deviation of the results.}
 \label{table}

\begin{figure}[h!]
    \centering
    \begin{subfigure}[t]{0.5\textwidth}
        \centering
        \includegraphics[height=2.in]{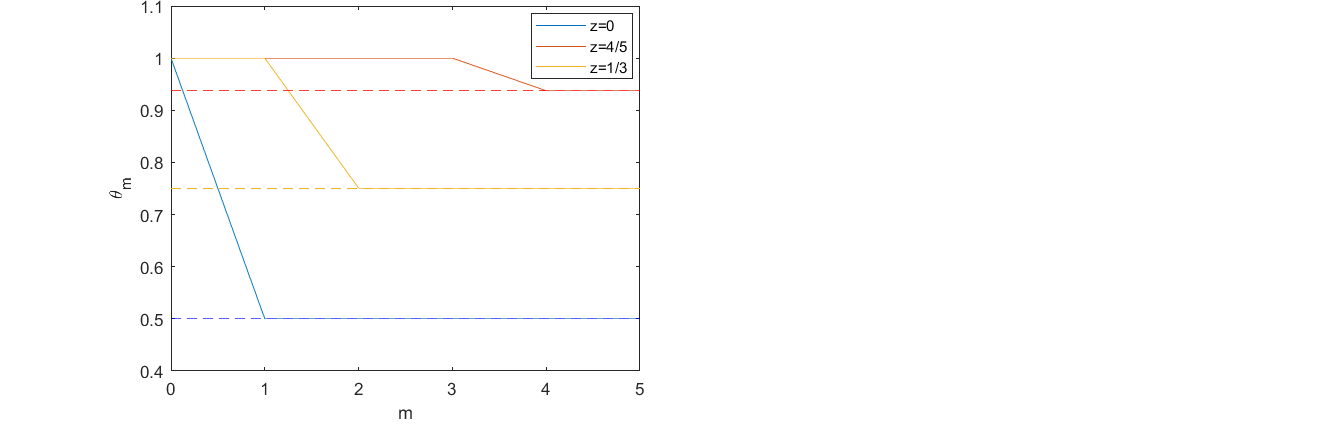}
        \caption{$2x\mod 1$}
    \end{subfigure}%
    ~
    \begin{subfigure}[t]{0.5\textwidth}
        \centering
     \includegraphics[height=2in]{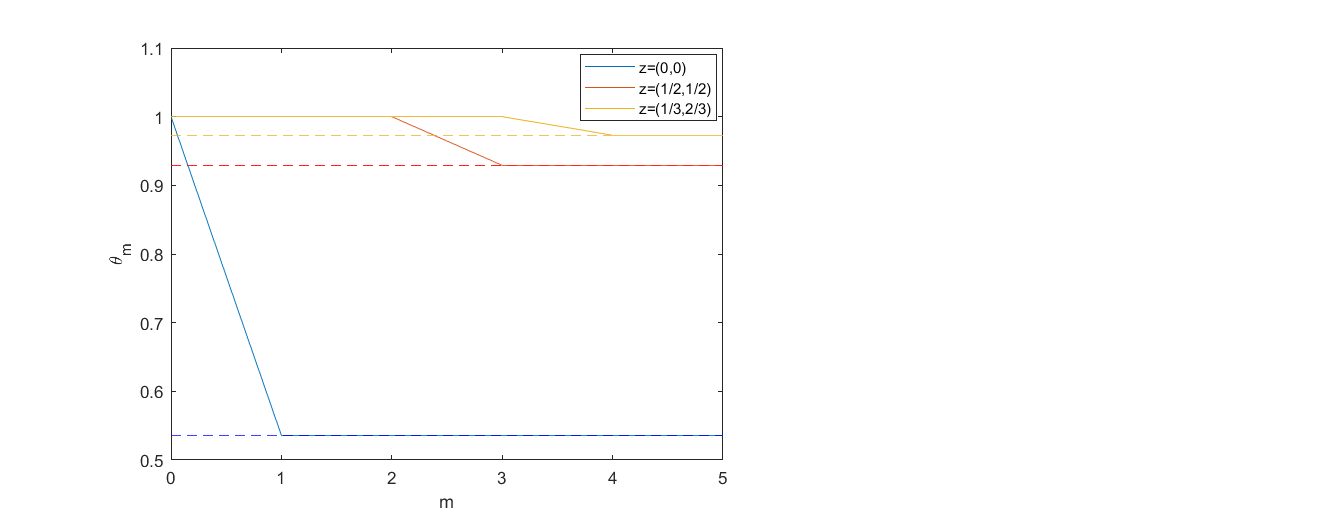}
         \caption{Cat's map}
    \end{subfigure}
   \caption{Evolution of $\hat\theta_m$ with $m$, for different target points and different maps. The numerical values are found in table \ref{table} and the parameter used in the text. The dashed line are the theoretical values of $\theta$.}
   \label{q}
\end{figure}

We observe in figure \ref{q} that for a target point of period $k$, $\hat\theta_m$ is equal to 1 for $m<k$, is equal to $\theta$ for $m \ge k$, due to the fact that $\hat q_{k-1}$ is non zero. For $m\ne k$, we also have $\hat q_{m-1} =0$, as expected.
Although the S\"uveges estimator is not in principle suitable to detect periodic points with period strictly larger than $1$, it is very often used to analyze time series since it is  particularly simple to implement. But there is a more interesting reason.
If the process is not periodic, for instance generated by random dynamical systems like those presented in the next sections, whenever the  S\"uveges estimator is strictly less than one, the same happens for the EI, thanks to our Proposition \ref{propo}.

Instead, as we saw in the table above, the difference with real situations of larger periodicity could be important. In such cases, the estimate $\hat\theta_m$ performs systematically better. We will use this estimator in the numerical computations of this paper.

 To evaluate the differences between the two estimates in high dimensional datasets, we test them on the North Atlantic daily sea-level pressure data  described in \cite{nature}. For each atmospheric state $z$, the EI index associated to the observable $-\log \text{dist}(z,.)$ was evaluated using the S\"uveges estimate and called the {\em inverse persistence} of $z$. The two estimations are shown in Figure \ref{climatebis}

 Estimates using $\hat\theta_5$ for these data are systematically lower than those with the order $0$ method (in average $0.065$ less), due to the contribution of some $\hat q_k$ for $k > 0$  (see the empirical distribution in figure \ref{climate}). In Figure \ref{climatebis} we present a scatter plot of the daily values of $\theta$ obtained with the two estimators. We see that there is a strong linear relation between the two estimates with an offset of about 0.1 days$^{-1}$. The estimation of the cross-correlation coefficient, namely the zeroth lag of the normalized covariance function, yields 0.94, meaning that the information contained in the two estimators is the same except for a restricted set of states $z$. Moreover, if we remove the time averaged values $\langle \hat\theta_{Su}\rangle$ of the S\"uveges estimator and  $\langle \hat\theta_5\rangle$ of the new estimator from the time series, we can compare the distribution of $\hat\theta_{Su}-\langle \hat\theta_{Su}\rangle$ with that of $\hat\theta_5-\langle \hat\theta_5\rangle$. For this comparison, we have used  a two sided Kolmogorov-Smirnov test and  verified  the null hypothesis that the estimators are from the same continuous distribution at the 5$\%$ confidence level.

%
%
%

The lowest differences between the two estimators are found for patterns close to the average field $z$. This corresponds to anomalies close to 0 hPa in Figure~(\ref{climatebis}b). Instead, large differences between estimators correspond to a peculiar $z$ pattern, consisting of a deep low pressure anomaly over the north of the domain and an anticyclonic anomaly over the southern part of the domain. This pattern resembles those observed in \cite{nature} for the minima of the local dimensions $d(z)$. This analysis confirms that the two estimators are different even beyond the simple shift in the values. The patterns where the two estimators mostly diverge can be thought as the pre-images of higher order periodic points of the attractor underlying the mid-latitude atmospheric circulation. In this sense, the combined use of both $\hat\theta_{Su}$ and $\hat\theta_5$ estimators could be useful to detect these special points of the dynamics in climate data and other natural phenomena.

\begin{figure}[h!]
        \centering
        \includegraphics[height=2.in]{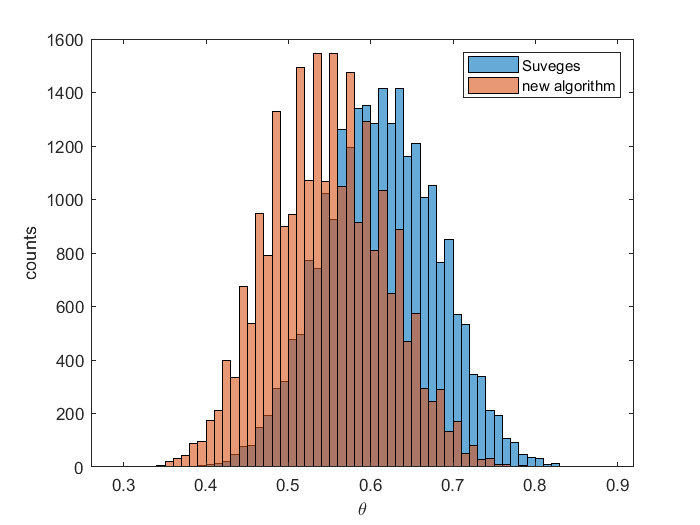}
        \caption{Comparison on the distribution of $\theta$ with the two estimates. The blue distribution is obtained using $\hat\theta_{Su}$ and the red one using $\hat\theta_5$ (for both taking as a threshold the 0.99-quantile of the observable distribution).}
   \label{climate}
\end{figure}

\begin{figure}[h!]
        \centering
        \includegraphics[width=0.9\textwidth]{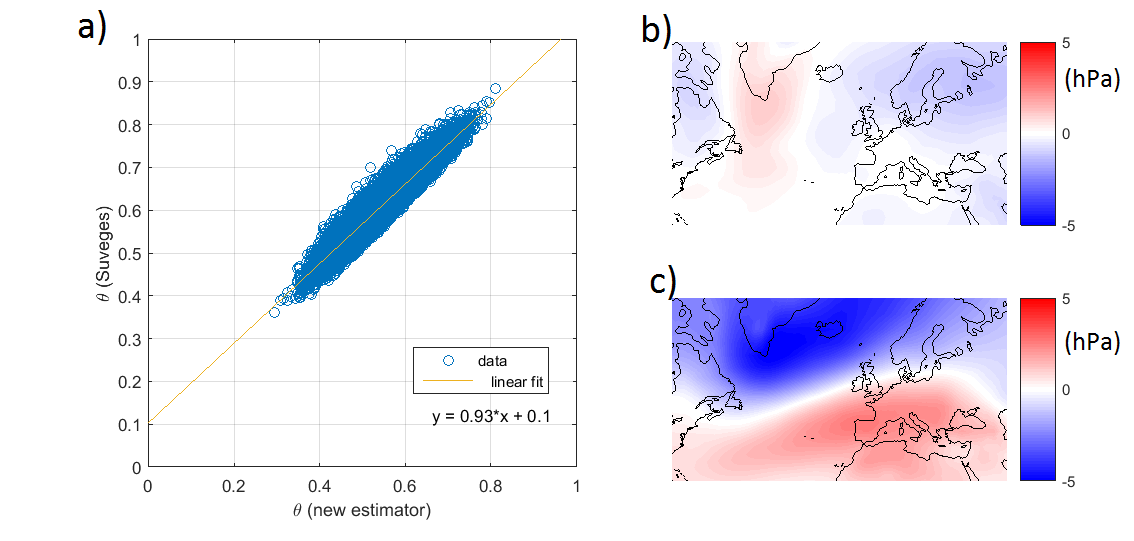}
        \caption{a) Scatter plot of $\hat\theta_{Su}$ (the S\"uveges estimator) vs $\hat\theta_5$ (the new estimator introduced in this work). Average map of the 5$\%$ sea-level pressure patterns such that the residual between $\hat\theta_{Su}$ and  $\hat\theta_5$ are smaller than the 5\% percentile (b) or larger (c) than the 95\% percentile, for both taking as a threshold the 0.99-quantile of the observable distribution.}
   \label{climatebis}
\end{figure}

%
%

\section{The random case: non-stationary situations (quenched noise)}
\label{nonstat}

The dynamics and the consequent detection of the extremal index could be affected in different ways. We begin with the dynamics by distinguishing first two cases:
\subsection{Sequential systems}
These systems are defined by concatenating maps chosen in some set, usually in the close neighborhood of a given map. As a probability measure one usually takes some ambient measure like Lebesgue, which makes the process given by the sequences of concatenations {\em non-stationary}. The extremal value theory and the extremal index  must be redefined (see \cite{FFV}, in particular Eq. (2.8)). In section 4.5 of the aforementioned paper,  we showed an example of a  sequential system modeled on maps chosen in the neighborhood $\mathcal{U}$ of a given  $\beta$-transformation $T_{\beta}$. According to suitable choices  of $\mathcal{U}$ it is possible to show either that the EI of the unperturbed map  and of the sequential system are the same, or the two  differ, in particular when the elements of the concatenation are far enough  from $T_{\beta}$. In this case the EI is simply $1$. Sequential  systems are a model for non-autonomous physical systems that do not admit a natural stationary probability measure. Statistical limit theorems can be proved as well in this context provided the observables are suitably centered (see \cite{CR, NV} for more details). The previous example with $\beta$-transformations suggests that when the maps are far enough from each other, the EI converges to $1$.

To the best of our knowledge, there are not numerical investigations of extreme value behaviors for sequential systems. We propose one of them in the Appendix.

\subsection{Random fibred systems} The sequential systems are in some sense too general, since there are no real prescriptions on the choice of the maps. A much more realistic class of non-autonomous systems is given by  random  fibred systems. They are constructed by taking a  driving map $\sigma$ that preserves a probability measure $\nu$ on the measurable space $\Omega$, and  which codes a family of transformations $f_{\omega},$ for $\omega \in \Omega$ on the fiber $X$ {\em via} the composition rule
        \begin{equation}\label{FFF}
        f^n_{\omega}:x=f_{\sigma^{n-1}\omega}\circ\cdots\circ f_{\omega}.
        \end{equation}
        The driving system encodes the external influences
on the system of interest: it acts deterministically in the choice of the evolution transformations on the fibers. As L.  Arnold wrote at the beginning of his monograph \cite{LA}: 

\begin{center} {\em Imagine a mechanism which at each discrete time $n$ tosses a (possibly complicated, many-sided) coin to randomly select a mapping $\phi_n$ by which a given point $x_n$ is moved to $x_{n+1}=\phi(x_n).$ The selection mechanism is permitted at time $n$ to remember the choices made prior to $n$, and even to foresee the future. The only assumption made is that the same mechanism is used at each step. This scenario, called {\em product of random mapping}, is one of the prototypes of a random dynamical systems.}
\end{center}

The mechanism used at each step is just the driving system. It is interesting to observe for the objectives of this paper, that random fibred systems have been used to analyze the  transport
phenomena in non-autonomous dynamical systems, such as geophysical
flows (see the enlightening review article \cite{CGT}). In order to perform statistics on these systems, we notice that  a family of sample measures $\mu_{\omega}$ lives on the fiber, which verifies the quasi-invariant equation $(f_{\omega})_{*}\mu_{\omega}=\mu_{\sigma \omega}$, where $(f_{\omega})_{*}$ is  the push-forward of the measure.   These sample measures will be taken as the probability measures that describe the statistical
properties along the fiber and they do not give rise to stationary processes. The relation between the sample measures and $\nu$ is that the measure $\mu:=\int \mu_{\omega}d\nu(\omega)$ is preserved by the skew (deterministic)  transformation $F(x, \omega)=(f_{\omega}(x), \sigma \omega),$ acting on the product space $X\times \Omega.$ The extremal index can be  suitably defined and it could be different from $1$ (see Corollary 5.1 and Theorem 5.3 in \cite{FFV} dealing with {\em random subshift}).
Transparent examples are constructed in the following  subsections.

\subsubsection{Random Lasota-Yorke maps}\label{CIR} \cite{FFMV}. Take
$\Omega=I^{\mathbb{Z}},$ where  $I=(1,\dots,m)$ is a finite alphabet with $m$ letters.
We associate to each letter  a piecewise expanding  map of the interval $f_{\omega}$  satisfying some smoothness and distortion standard assumptions. The map $\sigma$ is therefore the bilateral shift and $\nu$ any
ergodic shift-invariant non-atomic probability measure, for instance the Bernoulli measure with weights $p_1, \dots, p_m$. The sequence of concatenations is given by $f^k_{\omega}=f_{\omega_k}\circ\cdots\circ f_{\omega_1}$, with $\omega_j\in I$ and $j=1,\dots,k$ are the first $k$ symbols of the word $\omega$.\footnote{Notice that this is equivalent to Eq. (\ref{FFF}) by defining the map $\omega\rightarrow f_{(\omega)'}$, where $(\omega)'$ is the $\omega_1$ coordinate of $\omega$.} In this setting it can proved that the sample measures are equivalent to the  Lebesgue measure $m$ and for $m$-almost any target point $z\in X$ and $\nu$ almost any realization  $\omega \in \Omega$, the EI is equal to $1$ (see \cite{FFMV}). This corresponds to a quenched result since it depends on the choice of the sequences $\omega$.

\subsubsection{Rotations} \label{ROT}Another more physical quenched example is constructed in the following way. Let us take as $\Omega$ the unit circle $\mathbb{S}^1$ and as $\sigma$  the irrational rotation: $\sigma(\omega)=\omega+\alpha-\mathrm{mod}1$, $\omega\in \mathbb{S}^1$ with $\alpha\in \mathbb{R}$. Then we define $T(x)=3x-\mathrm{mod}1$ and make the correspondence:
\begin{align*}
\omega \rightarrow f_{\omega} \ \text{such that} \  f_{\omega}(x)=T(x)+\omega -\mathrm{mod}1.\\
 \sigma \omega \rightarrow f_{\sigma \omega}\  \text{such that} \  f_{\sigma \omega}(x)=T(x)+\sigma \omega- \mathrm{mod}1.\\
&\vdots\\
\sigma^k \omega \rightarrow f_{\sigma^k \omega}\  \text{such that} \ f_{\sigma^k \omega}(x)=T(x)+\sigma^k \omega-  \mathrm{mod}1.
\end{align*}
and so on, by composing after that as: $f_{\sigma^k \omega}\circ\cdots f_{\omega}$.

This last example was implemented numerically using the estimate $\hat\theta_5$ for different trajectories and choices of $\alpha$, either irrational or rational. We found an extremal index equal or very close to $1$ in all cases (see table \ref{table2}).\\

\begin{remark}
In the preceding two examples we considered the extremal index, but we did not indicate any method to compute it. We are in fact outside the stationary regime which provided us with the various formulae described in the previous  sections. Actually a definition of the EI in the non-stationary case is given in formula (2.8) in \cite{FFV}. This formula relies also on a suitable definition of the boundary level $u_n$ given in Eq. (2.2) in the aforementioned paper. The above formula (2.8) is a slight modification of the O'Brien formula and it could be reasonably recovered numerically by using Eq. (\ref{MC}) or the S\"uveges formula when we are not in presence of clustering. Although the EI was computed rigorously for the Random Lasota-Yorke maps as indicated in section \ref{CIR}, we do not dispose of a rigourous argument for the rotations and the numerical results we presented have been obtained with the $\hat\theta_5$ estimate. We observe that we could {\em assume} as a definition of the extremal index in these non-stationary situations the reciprocal of the expectation given by the distribution of the number of visits (see Eq. (\ref{REI}) in section 6). We will show in section \ref{roto} that the value for the EI computed in that way is the same as that provided by the $\hat\theta_5$ estimate.
\end{remark}


\begin{center}
\begin{tabular}{|c|c|c|c|c|}

  \hline
  Trajectory &  $\alpha=1/\pi$  & $\alpha=\sqrt2-1$ & $\alpha=4/5 $  \\
  \hline
  1 & 0.9995 & 0.9995 & 1\\
  2 & 0.9995 & 0.9996 & 1\\
  3 & 0.9997 & 0.9998 & 1\\
  4 & 0.9997 & 0.9995 & 1\\
  5 & 0.9996 & 0.9996 & 1\\
  6 & 0.9994 & 0.9996 & 1\\
  \hline

\end{tabular}

\captionof{table}{Values of $\theta$ found for different values of $\alpha$ and for different trajectories of length $10^7$. We used the estimate $\hat\theta_5$ and the $0.999-$quantile of the observable distribution as a threshold.}
 \label{table2}
\end{center}

\section{The random case: stationary situations (annealed noise)}
The preceding two situations dealt with non-stationary processes. We now focus on stationary processes given (i)  by   i.i.d. randomly chosen transformations, and (ii) by i.i.d. moving target sets.
\subsection{I.i.d. random transformations}
We now suppose that the maps $f_{\omega}$ are no longer driven by the measure-preserving map $\sigma$, but are chosen in an i.i.d.  manner.

\subsubsection {Additive noise} \label{AN}  A common way to build such a process is to choose a map $f$ and to construct the family of maps $f_{\xi}=f+\epsilon \xi$, where $\xi$ a random variable sampled from some distribution $\mathbb{G}$. This distribution can be the uniform distribution on some small ball of radius $\epsilon$ around $0$,  where $\epsilon$ is the intensity of the noise.  The iteration of the single map $f$ is now replaced by the concatenation $f_{\xi_n}\circ \cdots \circ f_{\xi_1}$,  where the $\xi_k$ are i.i.d. random variables with distribution $\mathbb{G}$. If $T$ has good expanding or hyperbolic properties, it is possible to show the existence of the so-called {\em stationary measure} $\rho_s$, verifying for any real bounded function $g$: $\int g d\rho_s=\int \int g\circ T_{\xi}d\rho_sd\mathbb{G},$ (see \cite{book}
Chap. 7, for a general introduction to the matter). If we now  take the probability product measure $\mathbb{Q}:=\mathbb{G}^{\mathbb{N}}\times \rho_s,$ any process of type $\{g(f_{\xi_n}\circ \cdots \circ f_{\xi_1})\}_{n\in \mathbb{N}}$ is $\mathbb{Q}$-stationary. The extremal index is obtained again by applying Eq. (\ref{TH}) where the sets  $\Omega_n^{(k)}(z)$ and $\mu(B(z, e^{-u_n}))$ are now weighted with the measure $\mathbb{Q}$.
The expectation being taken with respect to $\mathbb{G}$ too, makes this an {\em annealed} type random perturbation, while the fibred perturbation described above is a {\em quenched} one, the expectations depending on the realization $\omega.$  In the article \cite{AFV} we rigorously proved that for a large class of piecewise expanding maps of the interval perturbed with additive noise with the uniform distribution $\mathbb{G}$, the extremal index is $1$ as soon as $\epsilon$ becomes positive. The same happens for other kinds of maps and smooth distributions, as we showed in \cite{FFFFF} almost numerically. We believe that this mostly happens for the systems perturbed with i.i.d. noise, since the latter destroys all the periodic points. Nevertheless it is possible to construct annealed examples giving an EI less than one. This will be done in the next section.

\subsubsection{Discrete noise} \label{DN}

Let us  consider:
 \begin{itemize}
 \item Two maps $f_0=2x-\mathrm{mod} 1$ and $f_1=2x+b -\mathrm{mod}  1$, $0<b<1.$
     \item A {\em fixed} target set around $0$: $H_{n}(0):=B(0, e^{-u_n}),$ the ball of radius $e^{-u_n}$ around $0$. Notice that $0$ is a fixed point of $f_0$, but not of $f_1$.

 \end{itemize}

The indices $\xi_1, \dots, \xi_n$ in the concatenation $f_{\xi_n}\circ \cdots \circ f_{\xi_1}$ are  i.i.d. Bernoulli random variables with distribution, for instance, $\mathbb{G}=p\delta_{0}+(1-p)\delta_{1},$  with $p=\frac12$. The spectral theory developed in \cite{AFV} applies as well to this case, giving an absolutely continuous stationary measure $\rho_s$. Proposition 5.3 in \cite{AFV} gives for $q_0$:
\begin{equation}\label{QCC}
q_0= \lim_{n \rightarrow \infty}\frac{\int d\mathbb{G} (\xi) \rho_s(H_{n}(0)\cap f^{-1}_{\xi}H_{n}(0))}{\rho_s(H_{n}(0))}=\lim_{n \rightarrow \infty}\left[\frac12\frac{ \rho_s(H_{n}(0)\cap f^{-1}_{0}H_{n}(0))}{\rho_s(H_{n}(0))}+\frac12\frac{ \rho_s(H_{n}(0)\cap f^{-1}_{1}H_{n}(0))}{\rho_s(H_{n}(0))}\right]
\end{equation}
Since $0$ is not a fixed point of $f_1$, a standard argument gives $0$ to the limit in the square bracket on the right hand side. Instead by the mean value theorem applied to the term in the left square bracket we get in the limit of large $n$:
$q_0=\frac14$.

\begin{remark}\label{IR}
The preceding example could be interpreted as a perturbation of the map $f_0=2x-\mathrm{mod} 1$ around its fixed point $0$. Instead we change the point of view and consider it as a perturbation of the map $f_1= 2x+b-\mathrm{mod} 1$ around $0$. Supposing the map $f_0$ is chosen with probability $p_0$, we see that when $p_0=0$ (absence of perturbation), the $EI=1$, since $0$ is not a fixed point for $f_1$, but it jumps to $EI\le 1- 0.5\ p_0$ as soon as $p_0>0$ and regardless of the value of $b$. This  behavior is specular of what happened for the additive  noise described above: there it was enough to switch on the noise, no matter of its magnitude, to make the $EI$ equal to one. Here, the $EI$ changes in term of the probability of appearance of the perturbed map, no matter of its topological distance from the unperturbed one, the value of $|b|$ in this case. However, it is easy to make the $EI$ depending on the {\em distance} between the maps. For instance, let us take
 $f_0$ as $f_0=(2+j)x- \mathrm{mod} 1$, $j\in \mathbb{N}$, which could be seen as a strong perturbation.  By repeating the argument above we get  that the $EI\le 1-p_0\ \frac{1}{2+j}$.
\end{remark}
%
%
%
%
%
%
%
%
The preceding  computation of $q_0$ is valid for all values of $b \in (0,1)$. The question we may ask now is whether other $q_k$ are non zero. By a similar reasoning as for the computation of $q_0$, this problem is equivalent to ask for the existence of a concatenation $f_{\xi_{k+1}} \circ \cdots \circ f_{\xi_{1}}$ returning the point $0$ to itself after $k+1$ iterations and not before. Indeed, in that case we have that

 \begin{equation}
 q_k \ge \lim_{n \rightarrow \infty}\frac{\mathbb{G} (\xi_1,\dots,\xi_{k+1}) \rho_s(H_{n}(0)\cap (f^{-1}_{\xi_1}H_{n}(0))^c\cap \cdots \cap (f^{-k}_{\xi_k}H_{n}(0))^c \cap f^{-k-1}_{\xi_{k+1}}H_{n}(0))}{\rho_s(H_{n}(0))}= \mathbb{G} (\xi_1,\dots,\xi_{k+1}) )/2^k.
\end{equation}

 If the random orbit will not return to $0$, a standard argument gives the value $0$ for all $q_k$ when $k \neq 0$.

 To investigate the existence of such a concatenation, let us distinguish between the cases when $b$ is irrational or rational.

 If $b$ is irrational and $\xi_1 \neq 0$, any concatenation of any length maps $0$ into an irrational point, and so the orbit cannot come back to $0$ after having left it. For this reason, all the $q_k$ but $q_0$ are equal to $0$.

We investigate the case where $0<b=p/q<1$ is rational, where $p,q \in \mathbb{N}$ are mutually prime.


  We start by showing by induction that for $n\ge1$ and for all $i\in [2^{n}, 2^{n+1}-1]$, the random orbit starting from $p/q$ can attain the points $\frac{ip}{q}-\mathrm{mod}1$ (and only them) in exactly $n$ steps.

  This proposition is true at rank $n=1$, since $f_0(p/q)=2p/q$,   $f_1(p/q)=3p/q-\mathrm{mod}1$. Suppose now that it holds at rank $n$. Then we have that

\begin{itemize}

\item $f_0(2^{n}p/q-mod1)=2^{n+1}p/q-\mathrm{mod}1$\\
\item $f_1(2^{n}p/q-mod1)=(2^{n+1}+1)p/q-\mathrm{mod}1$\\
  \vdots\\
\item$f_0((2^{n+1}-1)p/q-mod1)=(2^{n+2}-2)p/q-\mathrm{mod}1$\\
\item$f_1((2^{n+1}-1)p/q-mod1)=(2^{n+2}-1)p/q-\mathrm{mod}1,$
  \end{itemize}
 and the proposition is true at rank $n+1$. This result implies that any point of the form $jp/q-\mathrm{mod}1$, with $j\in \mathbb{N}, j \ge 2$ can be attained by the random orbit starting from the point $p/q$.

 Let us consider $f_{\xi_{m}} \circ \cdots \circ f_{\xi_{2}} \circ f_{1}$ the concatenation of minimal length $m\ge2$ such that $f_{\xi_{m}} \circ \cdots \circ f_{\xi_{2}} \circ f_{1}(0)=qp/q-\mathrm{mod}1=0$. Note that the first map involved in this concatenation is $f_1$, so that the first iterate gives $p/q$. By the preceding argument, there exists such a concatenation and as we consider the smallest of the concatenations having this property, the random orbit leaving $0$ does not come back to $0$ before $m$ iterations. As discussed earlier, this implies a non zero value of $q_{m-1}$.

 In fact, there exists infinitely many non zero $q_k$: consider the concatenation of minimal length $m' \ge 2,$ $f_{\xi'_{m'}} \circ \cdots\circ f_{\xi'_2} \circ f_{1}$ such that $f_{\xi'_{m'}} \circ \cdots\circ f_{\xi'_2} \circ f_{1}(0)=(q+1)p/q-\mathrm{mod}1=p/q-\mathrm{mod}1$. By the argument in the induction, we have that either $m' = m$ or $m'=m+1$. Then we see that
 $$f_{\xi_{m}} \circ \cdots \circ f_{\xi_2} \circ f_{\xi'_{m'}} \circ \cdots \circ f_{\xi'_2}\circ f_{1}(0) =f_{\xi_{m}} \circ \cdots \circ f_{\xi_2}(p/q)=0.$$
 This proves the existence of a concatenation returning the point $0$ to itself after exactly either $2m$ or $2m+1$ iterations (and not before). Again we have proved that either $q_{2m}$ or $q_{2m-1}$ is non zero. Applying the same reasoning, we can prove that infinitely many $q_k$ have positive values.


 An interesting situation happens when $b=1/2$. In this case, for all $k$, $0$ can come back to itself after exactly $k+1$ iterations, and this happens only for the concatenation $f_1 \circ \cdots \circ f_1 \circ f_0$, where the map $f_1$ is applied succesively $k$ times. This sequence having probability $(1-p)^kp=1/2^{k+1}$, it is easily seen (by the very definition of $q_k$ in Eq. (\ref{QK}) and analogously to the proof for the computation of $q_0$), that  the $q_k$ are non zero and are equal to:

    $$q_k=\frac1{2^{k+1}}\lim_{n \rightarrow \infty}\frac{ \rho_s(H_{n}(0)\cap (f^{-1}_{1}H_{n}(0))^c\cap ... \cap  (f^{-k}_{1}H_{n}(0))^c \cap f^{-k-1}_{0}H_{n}(0))}{\rho_s(H_{n}(0))} = \frac1{4^{k+1}}.$$

In this situation, $\theta$ can be computed analytically and we find:

$$\theta=1-\sum_{k=0}^{\infty}(\frac14)^{k+1}=2/3.$$

We computed in table \ref{tableq} some estimates $\hat q_k$ and we find a very good agreement between theoretical and numerical estimates.

\begin{center}
\begin{tabular}{|c|c|c|c|c|c|}

  \hline
     &  Theoretical values  & Estimate $\hat q_k$ & Uncertainty  \\
  \hline
$q_0$ & 0.25 & 0.2500 & 0.0020\\
$q_1$ & 0.0625 & 0.0624 &  0.0016\\
$q_2$ & 0.015625  & 0.0156 &   $7.9.10^{-4}$ \\
$q_3$ &  0.00390625 &  0.0039 &   $3.94.10^{-4}$\\
$q_4$ & $9.765625.10^{-4}$ & $9.321.10^{-4}$ & $1.69.10^{-4}$\\
  \hline

\end{tabular}\\

\captionof{table} {Estimations of the $q_k$ using $\hat q_k$ for the precedent example with  $b=1/2$. These results compare favorably with theoretical values to a precision of the order of $10^{-4}$. Results are averaged over 20 trajectories of $5.10^6$ points and we took the 0.999-quantile of the observable distribution as a threshold. The uncertainty is the standard deviation of the results.}
 \label{tableq}

\end{center}

\subsection{Moving target}\label{MMTT}
Another source of disturbance comes from the uncertainty of fixing the target ball for the hitting times of the orbit. A more general theory of random transformations and moving balls will be presented elsewhere. Here we simply consider the case of a deterministic dynamics and a  target set shifted by a random displacement and we  assume an annealed approach.

\subsubsection{Discrete noise}
Let us take a point $z\in X$; then we consider  $\sigma$  the one-sided shift on $m$ symbols and set $\Omega=\{1,\cdots,m\}^{\mathbb{Z}}$. To each symbols,  $j, j=1, \dots, m,$  there corresponds a $z(j)\in [z-\epsilon, z+\epsilon]$.
As a probability $\mathbb{G}(\omega), \omega=(\omega_{0}, \omega_{1}, \dots)\in \Omega$, we put the  Bernoulli measure of weights $p_i=\frac1m, i=1,\dots,m$. It is invariant under the shift. Note that $\epsilon$ is fixed and is the strength of the  uncertainty.
At each temporal step  the map $T$ moves $x$ and the target point moves as well around $z$  under the action of $\sigma$. We now construct the direct product $S(x, \omega):=(T(x), \sigma\omega)$. The probability measure
 $\mathbb{P}:=\mu \times \mathbb{G}$ is invariant under $S$. As in section 2 we are interested in  the distribution of the random variable
\begin{equation}\label{HH}
H^r_n(x):=\text{\{First time the iterate $T^n(x)$ enters the ball $B(z(\sigma^n \omega)', e^{-u_n})$\}},
\end{equation}
where $\omega'=\omega_0$ is the first coordinate  of $\omega$. We define  the {\em usual} observable:
$$
\phi(x, \omega)= -\log |x-z(\omega')|=-\log |x-z(\omega_0)|,
$$
and notice that for all $k\ge 1:$
$$\phi(S^k(x,\omega))=-\log |T^k(x)-z((\sigma^k\omega)')|=-\log |T^k(x)-z(\omega_k)|.$$
The distribution of $H^r_n(x)$ is given by $\mathbb{P}\left(M_n\le u_n\right),$ where

$$M_n(x, \omega)=\max\{\phi(x, \omega'), \phi(S(x, \omega)), \dots, \phi(S^{n-1}(x, \omega))\},$$

and the boundary levels $u_n$ verify
$$
n \mathbb{P}(\phi(x, \omega)>u_n)=n \int \int d\mathbb{G}(\omega) \mu(B(z(\omega'), e^{-u_n}) \rightarrow \tau
$$
for some positive number $\tau$. We now choose a map $T$ on the unit interval admitting a spectral gap for the transfer operator and preserving an absolutely continuous invariant measure $\mu$. That operator acts, for instance, on the space of bounded variation function on $[0,1]$,  which can be enlarged to a Banach space by adding to the total variation, $|\cdot|_{TV}$,   the $L^1$ norm with respect to the Lebesgue measure $m$, $||\cdot ||_{L^1(m)}$.
We endow the space of  functions $g(x, \omega)$ defined on $X\times \Omega$ with the Banach norm defined by:
$$||g||_{\mathcal{B}}:=\int |g(\cdot, \omega)|_{TV} d\mathbb{G}(\omega)+||g||_{L^1(\mathbb{G}\times m)}.$$
One can show that with respect to this norm the transfer operator $\mathcal{L}_S$ for the direct product $S$  has a spectral gap on the largest eigenvalue $1$. We introduce now the perturbed operator $\tilde{\mathcal{L}}_S(g(x,\omega)=\mathcal{L}_S(g(x,\omega){\bf 1}_{B^c(z(\omega'), e^{-u_n})}(x))$, where $g\in \mathcal{B}$ is a function defined on the Banach space just introduced. We first notice that
$$
\mathbb{P}\left(M_n\le u_n\right)=\int\int  \tilde{\mathcal{L}_S^n}(h)dmd\mathbb{G},
$$
where $h$ is the density of $\mu$, which, with our assumptions, is bounded from below by the constant $C$.
Since $||(\tilde{\mathcal{L}}_S-\mathcal{L}_S)h||_{L^1(\mathbb{G}\times m)}\le C e^{-u_n}||h||_{\mathcal{B}}$,  we can apply the perturbation theory mentioned in section 2, (see \cite{GK}, \cite{cml}, \cite{book} ch. 7), and show that $\mathbb{P}\left(M_n\le u_n\right)$ converges to
the Gumbel law $e^{-\theta \tau}$.   The extremal index $\theta$ is given by the adaptation to the actual case of Eq. (\ref{QK}), with the $q_0$ which  reads:
\begin{equation}\label{NQK}
q_0=  \lim_{n\rightarrow \infty} \frac{\int d\mathbb{G} \ \mu\left(B(z(\omega'),n)\cap T^{-1}B(z(\sigma^{-1}\omega)',n)\right)}{\int d\mathbb{G}\  \mu(B(z(\omega'),n)}.
\end{equation}
We now give a simple example for which $q_0>0$, which implies that the extremal index is strictly less than $1$. Take as $T$ the map $T(x)=2x -\text{mod} 1$, and an alphabet of $4$ letters: $\{0,1,2,3\}$  with equal weights $1/4$. Moreover $\mu$ is the Lebesgue measure. Then  we set the associations:
\begin{itemize}
\item $0\rightarrow z_0$
\item $1\rightarrow z_1$
\item $2\rightarrow z_2$
\item $3\rightarrow z_3$,
\end{itemize}
where $z_i, i=0,1,2,3$ are points in the unit interval verifying the following assumptions:
\begin{itemize}
\item $T(z_1)=T(z_2)=z_0; \ T(z_0)=z_3.$
\item $T(z_3)\neq z_i, i=0,1,2,3.$

\end{itemize}
The numerator of $q_0$ is:
$$\int_{[\omega_0=1,
\omega_{-1}=0]}\mu\left(B(z_1,n)\cap T^{-1}B(z_0,n)\right)d\mathbb{G}+
$$

$$
\int_{[\omega_0=2,
\omega_{-1}=0]}\mu\left(B(z_2,n)\cap T^{-1}B(z_0,n)\right)d\mathbb{G}+
\int_{[\omega_0=0,
    \omega_{-1}=3]}\mu\left(B(z_0,n)\cap T^{-1}B(z_3,n)\right)d\mathbb{G}+$$

$$[\text{sum over all the other cylinders of length $2$}],$$
where $[\omega_0,\omega_{-1}]$
denotes a cylinder with fixed coordinates $\omega_0$ and $\omega_{-1}$.

For the denominator we get the value:
$$
\int d\mathbb{G} \mu(B(z(\omega'),n)=\sum_{l=1}^n \int_{[\omega_0=l]}d\mathbb{G} \mu(B(z(l),n)=\sum_{l=1}^4 p_l\ \mu(B(z(l),n),
$$
By using the same arguments as in Eq. (\ref{QCC}), we see that for all the cylinders in the numerator different from the three explicitly given  above, the integrals give zero in the limit of large $n$. For the first of the three cases above (the others being similar),  we get
$$\int_{[\omega_0=1,
\omega_{-1}=0]}\mu\left(B(z_1,n)\cap T^{-1}B(z_0,n)\right)d\mathbb{G}=\frac12\int_{[\omega_0=1,
    \omega_{-1}=0]}\mu\left(B(z_0,n)\right)d\mathbb{G}=\frac12 (\frac14)^2\mu\left(B(z_0,n)\right).$$

The numerator therefore contributes with $\frac32 (\frac14)^2\mu\left(B(z_0,n)\right)$.
The denominator is (by the Lebesgue translation invariance): $\mu\left(B(z_0,n)\right)$. Therefore we get $q_0= \frac32 (\frac14)^2=0.09375.$

To get $q_1$, we need to compute the following quantity:

\begin{equation}
q_1=\lim_{n\rightarrow \infty} \frac{\int d\mathbb{G}\  \mu\left(B(z(\omega'),n)\cap T^{-1}B(z(\sigma^{-1}\omega)',n)^c \cap T^{-2}B(z(\sigma^{-2}\omega)',n)\right)}{\int d\mathbb{G}\  \mu(B(z(\omega'),n)}.
\end{equation}

We start, as before, by summing the integral in the numerator over all the possible cylinders of length 3. Among  those cylinders, it is easy to check that only 6 of them contribute to the mass, namely $([3,1,1],[3,2,1],[3,3,1],[3,1,2],[3,2,2]$ and $[3,3,2])$. For the integral associated to the first cylinder, we have  (the five others are similar):

$$\int_{[3,1,1]} d\mathbb{G}\ \mu\left(B(z(\omega'),n)\cap T^{-1}B(z(\sigma^{-1}\omega)',n)^c \cap T^{-2}B(z(\sigma^{-2}\omega)',n)\right)=(\frac12)^2\int_{[3,1,1]} d\mathbb{G} \ \mu(B(z_3,n)$$

$$=(\frac14)^3(\frac12)^2 \mu(B(z_3,n))$$

By summing over the six cylinders and dividing by the denominator, we get the result:

$$q_1=6(\frac14)^3(\frac12)^2.$$

We tested this result numerically with different sets of target points following the aforementioned assumptions. We find good agreements of $\hat q_0$ with the theoretical value of $q_0=  0.09375$ and of $\hat q_1$ with the theoretical value of $q_1=0.0234375$ with a precision of order $10^{-3}$ (see table \ref{table}).

In the case where none of the points $z_i$ is the $k^{th}$ iterate of another (in particular none of the points $z_i$ is $k$-periodic), it can be shown that $q_{k-1}$ is equal to $0$ for $k>2$. This is indeed what we find in our numerical simulations, which were performed up to the order $5$.

\begin{center}
\begin{tabular}{|c|c|c|c|c|}
  \hline
  $z_0$ &  $\hat\theta_5$ & uncertainty & $\hat q_0$ & $\hat q_1$  \\
  \hline
  2/11 & 0.883 & 0.0017 & 0.0937 & 0.0231 \\
  10/13 & 0.883 & 0.0016 & 0.0936 & 0.0234 \\
  $1/\pi$ & 0.883 & 0.0022 & 0.0933 & 0.0234\\
  \hline

\end{tabular}\\

\captionof{table}{5-order estimates of $\theta$ found for different $z_0$ (the other points $z_1$,$z_2$ and $z_3$ are computed to satisfy the assumptions of the presented example). We used as a threshold the $0.999$-quantile of the observable distribution and the uncertainty is the standard deviation of the results. The values for $\hat q_0$ and $\hat q_1$ are averaged over the $20$ trajectories and match with the theoretical results. All the other $\hat q_k$ computed are equal to $0.$}
 \label{table}
\end{center}

This example is in some sense atypical since the four points $z_0,\dots,z_3$ could be far away from each other. For instance,  $z_1$ and $z_2$ are the two predecessors of $z_0$ and therefore they are on the opposite sides of $\frac12$.
By constraining the points to be in a small neighborhood of a given {\em privileged} center, the previous effects should be absent and the EI should be one, or very close to it. However our example shows that in the presence of moving target, it is not the periodicity which makes the EI eventually less than one. This gives another concrete example where the series in the spectral formula has at least two terms different from zero.

\subsubsection{Continuous noise}\label{CCNN}

We claimed above that when the center of the  target ball can take any value in a small neighborhood of  a given  point $z_0$, the EI collapses to $1$, even for periodic points $z_0$. To model this more physically realistic situation, let us fix $\epsilon>0$ and consider the set $Z_\epsilon=[z_0-\epsilon,z_0+\epsilon]$. We define a map $f$ acting on $Z_\epsilon$ with an associated invariant probability measure $\nu$ that drives the dynamics of the target point $z\in Z_\epsilon$. We  suppose that $\nu$ is not atomic.  The observable considered is now $\phi(x,z)=-\log |x-z|$ on the product space $\{X\times Z_\epsilon, \mu \times \nu\}$.  By similar arguments to the ones we described for the discrete perturbation of the target point, we can show  the existence of an extreme value law for the process $\phi \circ (T^k, f^k)$, with an EI  given by Eq. (\ref{QK}), with
\begin{equation}\label{proofq}
q_k=\lim_{n\rightarrow \infty} \frac{\int_{Z_\epsilon} d\nu(z)  \mu\left(B(z,n)\cap T^{-1}B(f(z),n)^c \cdots \cap T^{-k-1}B(f^{k+1}(z),n)\right)}{\int_{Z_\epsilon} d\nu(z)  \mu(B(z,n)},
\end{equation}
where $B(y,n)$ denotes a ball around $y$ of radius $e^{-u_n}$.
We  have:
\begin{proposition}\label{theo}
Suppose that for all $k\in\mathbb{N}$, $\nu(\{z\in Z_\epsilon |T^{k+1}(z)=f^{k+1}(z) \})=0$ and that $\mu$ is absolutely continuous with respect to Lebesgue with a bounded density $h$ such that $h\ge\iota>0$. Then the extremal index is $1$.
\end{proposition}
\begin{proof}
  We compute the term $q_k$ given by Eq. (\ref{proofq}) and denote
  $$Z_1^{n,k}=\{z\in Z_\epsilon | \text{dist}(z, T^{-k-1}B(f^{k+1}(z),n))>r_n\},$$

where $r_n=e^{-u_n}$. We can write the numerator in Eq. (\ref{proofq}) as a sum of integrals over $Z_1^{n,k}$ and its complementary set  $Z_2^{n,k}=Z_\epsilon \backslash Z_1^{n,k}$.

Since for $z \in Z_1^{n,k}$, $z$ and $z^*$ are at a distance larger than $r_n,$ for all $z^*\in T^{-k-1}B(f^{k+1}(z),n)$, we have that $B(z,n)\cap T^{-k-1}B(f^{k+1}(z),n)=\emptyset$, and the integral over $Z_1^{n,k}$ is zero for all $n$.

It remains now to treat the integral over $Z_2^{n,k}$. We have that:
\begin{equation}
\begin{aligned}
q_k &\leq \lim_{n\rightarrow \infty} \frac{\int_{Z_2^{n,k}} d\nu(z) \mu(B(z,n))}{\int_{Z_\epsilon} d\nu(z)\mu(B(z,n))}
      &\leq \lim_{n\to\infty} \nu (Z_2^{n,k}) \frac{\sup_{x\in M} \mu(B(x,n))}{\inf_{x\in M} \mu(B(x,n))}\\
      &= \lim_{n\to\infty} \nu (Z_2^{n,k}) \frac{\sup_{x\in M}h(x)}{\inf_{x\in M} h(x)}.
\end{aligned}
\end{equation}
With our assumptions on $h$, the fraction in the last term is finite. Let us consider the limit set $Z_2^k=\cap_n
 Z_2^{n,k}=\{z\in Z_\epsilon |T^{k+1}(z)=f^{k+1}(z) \}$. We have that $\lim_{n\to\infty} \nu (Z_2^{n,k})=\nu(Z_2^{k})=0$, by hypothesis and so $q_k=0$ for all $k$ and therefore $\theta=1.$
\end{proof}

\begin{remark}
In practice, the requirement that $\nu(\{z\in Z_\epsilon |T^{k+1}(z)=f^{k+1}(z)\})=0$ for all $k$ is satisfied for a large class of situations. We take two maps $T$ and $f$ intersecting at a countable number of points $Z_2^1$. We suppose that the preimage of any point by the applications $T$ and $f$ is a countable set. Then the set $Z_2^k=T^{-k-1}Z_2^1 \cup f^{-k-1}Z_2^1$ is  countable as well   for all $k\in \mathbb{N}$. Since $\nu$ is non-atomic, we have that  $\nu(\{z\in Z_\epsilon |T^{k+1}(z)=f^{k+1}(z)\})=0$ is verified and the proposition holds.
\end{remark}

\subsubsection{Observational noise}
 Instead of considering a perturbation on the target set driven by the map $f$, we consider the center of the target ball as  a random variable uniformly distributed in the neighborhood of size $\epsilon$ of the point $z_0$.  This situation  models  data with uncertainty or disturbances in their detection, and is equivalent to the case  {\em observational noise} which we considered in \cite{obs} in the framework of extreme value theory. In this approach, the dynamics of a point $x$ is given by $T^kx+\xi_k$, where $\xi_k$ is a random variable uniformly distributed in a ball centered at $0$ of radius $\epsilon$. The process  is therefore  given by $-\log(dist(T^kx+\xi_k,z_0))$. When this is transposed in the  moving target case,  it becomes $-\log(dist(T^kx,z_0+\xi_k))$ and the two approaches define the same process, which give the same stationary distribution. In \cite{obs} we proved rigorously in some cases and numerically in others, that for a large class of chaotic maps, an extreme value law holds with an extremal index equal to $1$. By equivalence of the two processes, an extreme value law also holds in the moving target scenario and the EI is equal to $1$.

 We tested numerically the scenario of a target point following a uniform distribution in an interval of length $\epsilon$ centered in a point $z_0$, for several maps of the circle (Gauss map, $3x-\mathrm{mod}1$, rotation of the circle). For all of them, we find that for periodic $z_0$, the extremal index which is less than $1$ when unperturbed, converges to $1$ as the noise intensity $\epsilon$ increases (see an example for the $3x-\mathrm{mod}1$ map in figure \ref{periodic}-A). When $z_0$ is not periodic (we chose it at random on the circle), the EI is equal to $1$ for all $\epsilon$ (see again figure \ref{periodic}-B).

\begin{figure}[h!]
    \centering
    \begin{subfigure}[t]{0.5\textwidth}
        \centering
        \includegraphics[height=2.in]{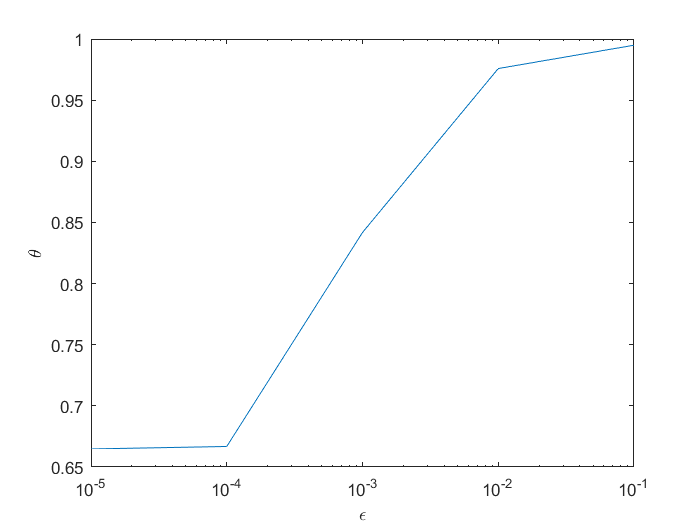}
        \caption{$z_0=1/2$ (fixed point)}
    \end{subfigure}%
    ~
    \begin{subfigure}[t]{0.5\textwidth}
        \centering
     \includegraphics[height=2in]{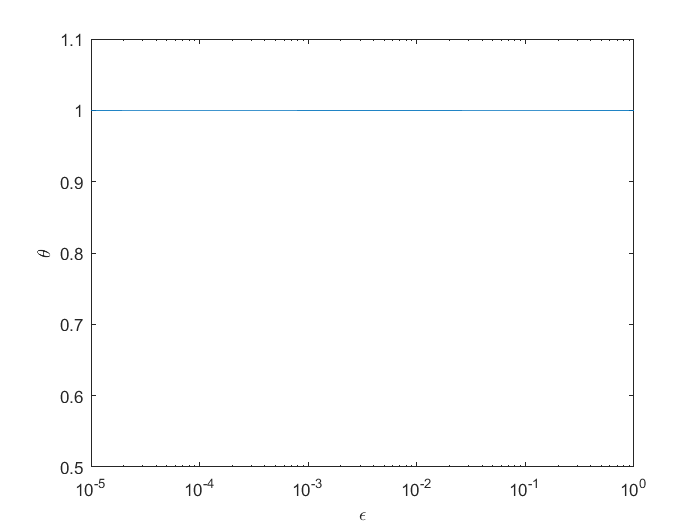}
         \caption{Generic $z_0$}
    \end{subfigure}
   \caption{Influence of the intensity of noise $\epsilon$ on the EI for the $3x-\mathrm{mod}1$ map perturbed by uniform noise. We simulated trajectories of $5.10^7$ points and took $p=0.999$. The computations are made using the estimate $\hat \theta_5$.}
   \label{periodic}
\end{figure}

\section{The dynamical extremal index}\label{SDEI}
Up to now the target set was a ball around a point. In the attempt to describe the synchronisation of coupled map lattices, Faranda et al. \cite{cml} introduced the neighborhood of the diagonal in the $n$-dimensional hypercube and defined accordingly the extremal index related to the first time the maps on the lattices become close together. Faranda et al. \cite{cml} then used that approach to define a new type of observable, first in  \cite{D2} and successively generalized in \cite{CFMVY}, which brought to a different interpretation of the extremal index, in terms of Lyapunov exponents instead of periodic behaviors.

Let us consider the $k-$fold ($k>1$) direct product $(X,\mu,T)^{\bigotimes k}$ with the direct product map $T_k=T\circ \cdots\circ T$ acting on the product space $X^k$ and the product measure $\mu_k=\mu \circ\cdots\circ \mu$.

 Let us now define the observable on $X^k$
\begin{equation}\label{OB}
\phi(x_1,x_2,\dots,x_k)=-\log(\max_{i=2,\dots,k}d(x_1,x_i)),
\end{equation}

where each $x_i \in X$. We also write $\overline{x}_k:=(x_1,x_2,\dots,x_k)$ and $T_k(\overline{x}_k)=(Tx_1,\dots,Tx_k)$.
As we explained in \cite{D2}, if we put
$M_n(\overline{x}_k)=\max \{\phi(\overline{x}_k),\dots,\phi(T^{n-1}_k(\overline{x}_k)\}$
and look at the distribution of the maximum $\mu_k(M_n \leq u_n)$,
this distribution is non-degenerate and converges to the Gumbel law for $n \to \infty$ $e^{-\theta_k \tau}$,
provided we can find a sequence $u_n \to \infty$ and verifying
$\mu_k(\phi>u_n) \to \frac{\tau}{n}$,
where $\tau$ is a positive number. The quantity $\theta_k$ is our new   extremal index and it will be described later on. Notice first that
\begin{equation}\label{GD}
\mu_k(\phi>u_n)=\int_{M^k} d\overline{x}_k \mathbf{1}_{B(x_1,e^{-u_n})}(\overline{x}_k)\dots\mathbf{1}_{B(x_1,e^{-u_n})}(\overline{x}_k)
                     =\int_M dx_1\mu (B(x_1,e^{-u_n}))^{k-1}.
\end{equation}

Our first motivation to investigate the observable in Eq. (\ref{OB}) was the fact that the integral on the right hand side of Eq. (\ref{GD}) scales like $e^{-u_n D_k(k-1)}$, where $D_k$ denotes the generalized dimension of order $k$ of the measure $\mu$.
We now describe $\theta_k$. We first define:
\begin{equation}\label{deltak}
\Delta_n^k=\{(\overline{x}_k), d(x_1,x_2)<e^{-u_n},\dots,d(x_1,x_k)<e^{-u_n}\}.
\end{equation}
By using the spectral technique in \cite{D2}, and the analytical results of \cite{cml},  it is possible to show that:
\begin{equation}\label{HHH}
\theta_k=1-\lim_{n \to \infty} \frac{\mu_k(\Delta_n^k \cap T_k^{-1}\Delta_n^k)}{\mu_k(\Delta_n^k)}.
\end{equation}
The quantity $\theta_k$ is the {\em dynamical extremal index} (DEI) appearing at the exponent of the Gumbel law.

For $C^2$ expanding maps of the interval, which preserve an absolutely continuous invariant measure $\mu=hdx$ with strictly positive density $h$ of bounded variation, it is possible to compute the right hand side  of Eq. (\ref{HHH}) and get:
\begin{multline}\label{20a}
\mu_k(\Delta_n^k \cap T_k^{-1}\Delta_n^k)=
\int dx_1h(x_1) \int dx_2h(x_2)\ \chi_{B(x_1,e^{-u_n)})}(x_2)\chi_{B(Tx_1,e^{-u_n)})}(Tx_2) \cdots \\
\cdots \int dx_k h(x_k)\chi_{B(x_1,e^{-u_n)})}(x_k)\chi_{B(Tx_1,e^{-u_n)})}(Tx_k).
\end{multline}
All the $k-1$ integrals above factorize, and depend on the parameter $x_1$. Therefore they can be treated as in the proof of Proposition 5.5  in \cite{cml}, yielding the rigorous result:
\begin{proposition}\label{prop-20}
Suppose that: the map $T$ belongs to $C^2$; it preserves an absolutely continuous invariant measure $\mu=hdx$, with strictly positive density $h$ of bounded variation; it verifies conditions $P1-P5$ and $P8$ in \cite{cml}\footnote{These conditions essentially ensure that the transfer operator associated with the map $T$ has a spectral gap and that the density $h$ has finite oscillation in the neighborhood of the diagonal.}. Then
\begin{equation}\label{EEII}
\theta_k=1-\frac{\int\frac{h(x)^k}{\mid DT(x)\mid^{k-1}}dx}{\int h(x)^kdx}.
\end{equation}
\end{proposition}

This formula uses the translational invariance of the Lebesgue measure: we refer to sections II-B and II-C in \cite{D2} for analogous extensions to more general invariant measures and to SRB measures for attractors. As remarked in \cite{D2}, whenever the density does not vary too much, or alternatively the derivative (or the determinant of the Jacobian in higher dimensions) are almost constant, we expect a scaling of the kind: $\theta_k \sim 1-e^{-(k-1)h_m},$
where $h_m$ is the metric entropy (the sum of the positive Lyapunov exponents).

In section D of \cite{D2} and for the case $k=2$, we replaced the iterations of a single map with concatenations of i.i.d. maps chosen with additive noise, see above, section 4.1. Although the dynamical extremal index is related to the Lyapunov exponents, it is also influenced by the fact that the set $\Delta=\cup_{i=2}^k \{x_i=x_1\}$ is invariant in the deterministic situation. By looking at Eq. (\ref{HHH}), we
see that we estimate the proportion of the neighborhood of
the invariant set $\Delta$ returning to itself. As argued in \cite{D2}, that
estimate gives information on the rate of backward volume
contraction in the unstable direction. Since the noise generally destructs these invariants sets, we expect the extremal
index be equal to $1$ or quickly approaching $1$ for all $k$ when the noise increases. The situation depends again on the characteristics of the noise.

\subsection{Additive noise} \label{AD} The prediction at the end of the preceding section is confirmed by concatenating maps perturbed with additive noise and smooth probability density function (see the numerical computations shown in figure \ref{thetaq}). As in section \ref{DN} the perturbation is of annealed type. For each map, trajectories of $10^7$ points were simulated and the $0.999-$quantile of the distribution of the observable $\phi$ was selected as a threshold. As only the $q_0$ term is non-zero, we choose the S\"uveges estimate for our computations, to be assured that no error coming from $\hat q_k$ terms of higher order are added.

\begin{figure}[h!]
    \centering
    \begin{subfigure}[t]{0.5\textwidth}
        \centering
        \includegraphics[height=2.in]{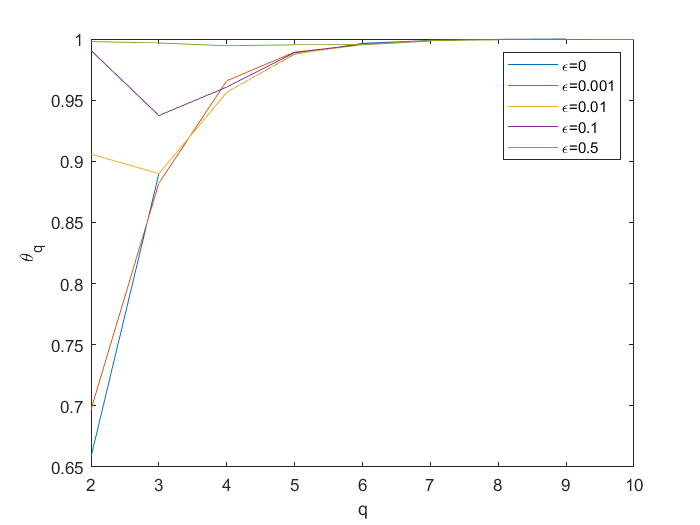}
        \caption{$3x-mod1$}
    \end{subfigure}%
    ~
    \begin{subfigure}[t]{0.5\textwidth}
        \centering
     \includegraphics[height=2in]{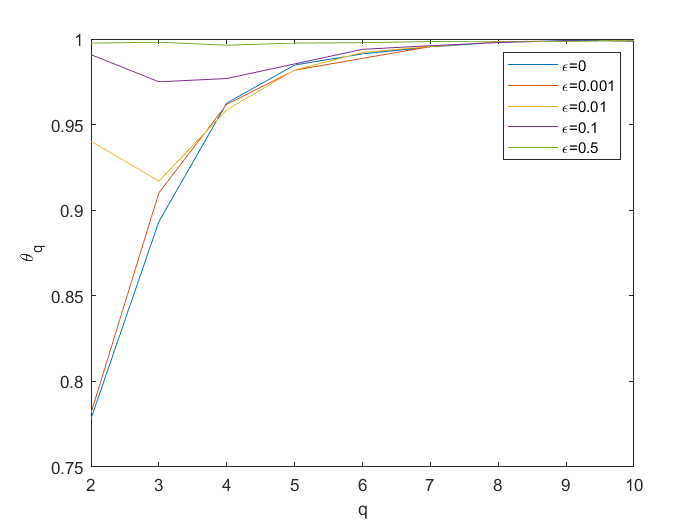}
         \caption{$\frac1x-mod1$ (Gauss map)}
    \end{subfigure}
   \caption{Influence of the intensity of noise $\epsilon$ on the $\theta_k$ spectrum for maps of the circle perturbed by uniform noise. We used the S\"uveges estimate and the 0.999-quantile of the observable distribution as a threshold, as mentioned in the text.}
   \label{thetaq}
\end{figure}

\subsection{Discrete noise}\label{DND}
A different scenario happens if we look at sequences of finitely maps chosen in an i.i.d. way according to a Bernoulli process as we did in section \ref{DN}. We choose now the  two maps $f_0(x)=3x-\mathrm{mod}1$ and $f_1(x)=3x+b-\mathrm{mod}1$ and we put the same distribution $\mathbb{G}$ defined in \ref{DN}. A combination of Eqs. (\ref{QCC}) and (\ref{HHH}) gives for the term $q_0$ entering the infinite sum defining the DEI:
\begin{equation}
q_0= \lim_{n \to \infty} \frac{\int_{\Omega}\rho_k (\Delta_n^k \cap f_{\xi^k}^{-1} \Delta_n^k)d\mathbb{G}^k(\xi^k)}{\rho_k(\Delta_n^k)},
\end{equation}
where: $\xi^k$ is the vector $\xi^k=(\xi_1, \dots, \xi_k)$, $f_{\xi^k}^{-1}=f_{\xi_1}^{-1}\cdots f_{\xi_k}^{-1}$ and $\rho_k$ is the stationary measure constructed as in \cite{AFV}. We now split the integral over the cylinders $\xi^{k,0}=[\xi_1=0, \dots, \xi_k=0], \xi^{k,1}=[\xi_1=1, \dots, \xi_k=1]$ and their complement. Therefore we have:
\begin{equation}\label{FT}
\frac{\int_{\Omega}\rho_k (\Delta_n^k \cap   f_{\xi^k}^{-1} \Delta_n^k)d\mathbb{G}^k(\xi)}{\rho_k(\Delta_n^k)}=\mathbb{G}(\xi^{k,0})
\frac{\rho_k (\Delta_n^k \cap f^{-n} _0 \Delta_n^k)}{\rho_k(\Delta_n^k)}+\mathbb{G}(\xi^{k,1})
\frac{\rho_k (\Delta_n^k \cap f^{-n}_1 \Delta_n^k)}{\rho_k(\Delta_n^k)}+
\end{equation}
\begin{equation}\label{ST}
\text{\{the integrals are computed over all the other cylinders of length $k$\}}.
\end{equation}

The two fractions in Eq. (\ref{FT}) are equal to the $q_0^{(\text{unp})}$ of the unperturbed systems, which are the same by the particular choices of the maps and by Eq. (\ref{EEII}).

All the other terms in Eq. (\ref{ST}) are zero, we now explain why.   Let us in fact  consider a vector $\tilde{\xi}^k$ different from $(0,0,\dots,0)$ and $(1,1,\dots,1)$.
Suppose that  the first coordinate $\tilde{\xi}^k_1$ of $\tilde{\xi}^k$ is zero and let $\tilde{\xi}^k_j\neq 0$ for some $1<j\le k$. If $x_k \in \Delta_n^k$, and since
$f_0 x_1=3x_1-\mathrm{mod}1$, $f_1 x_j=3x_j+b-\mathrm{mod}1$,  we have that $d(f_0 x_1  , f_1 x_i) \rightarrow_{n \to \infty} \min \{b,1-b\}$ as $d(x_1,x_i)<e^{-u_n}$.

 Therefore when $n$ is large enough, $\forall x_k \in \Delta_n^k$, $f^{\tilde{\xi}^k}x_k:= f^{\tilde{\xi}_k^k}\cdots f^{\tilde{\xi}_1^k} x_k\notin \Delta_n^k$ and $\Delta_n^k \cap f^{-1}_{\tilde{\xi}^k} \Delta_n^k = \emptyset$. In conclusion we have

\begin{equation}\label{ppp}
\theta_k=1-q_0^{(\text{unp})}(p^k+(1-p)^k),
\end{equation}
where $p$ is the weight of the Bernoulli measure. This last formula has been tested numerically with the maps $f_0$ and $f_1$, for different values of $b$ and with $p=0.5$.  Good agreement is found for different values of $b$ (see figure \ref{deiber}).
We use trajectories of $5.10^7$ points and the $0.995-$ quantile of the observable distribution as a threshold. Estimations are made using the S\"uveges estimate, for the same reasons as described earlier.

\begin{figure}[h!]
        \centering
        \includegraphics[height=2.in]{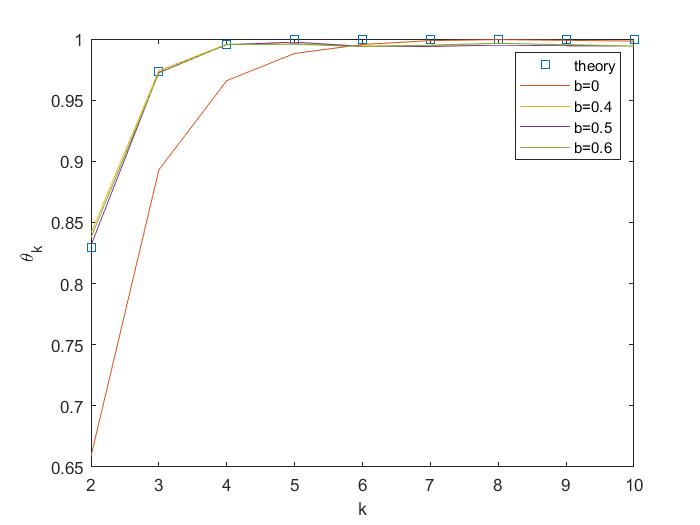}
        \caption{$\theta_k$ spectrum for the sequence of two maps described in the example, found using the S\"uveges estimate. The red curve is the value for $b=0$, and for different choices of $b\ne0$, the values of $\theta_k$ follow  the predictions of Eq. (\ref{ppp}) (blue squares).}
        \label{deiber}
    \end{figure}

\section{Point processes: statistics of persistence}
We said in section 2 that the extremal index is the inverse of the average cluster size when we look at exceedances  in the shrinking neighborhood of a given target point $z$. To be more precise let us start with the deterministic case.

\subsection{The deterministic case}
 We consider the following counting function
\begin{equation}\label{PP1}
N_n(t)=\sum_{l=1}^{\lfloor{\frac{t}{\mu(B(z,r_n))}}\rfloor}{\bf 1}_{B(z,r_n)}(T^l(x)),
\end{equation}
where the radius $r_n$ goes to $0$ when $n$ tends to infinity.
We are interested in the distribution
\begin{equation}\label{PP2}
\mu( N_n(t)=k), \ k\in \mathbb{N}.
\end{equation}
for $n\rightarrow \infty.$
It has been proved (see for instance \cite{HV, JF}) that when $z$ is not a periodic point, $\mu( N_n(t)=k)$ converge to the Poisson distribution $\frac{t^k e^{-t}}{k!}$, while for a periodic point of minimal period $p$ we have the Poly\`a-Aeppli distribution
\begin{equation}
 \mu (N_n(t)=k)\rightarrow e^{-\theta t}
\sum_{j=1}^k(1-\theta)^{k-j}\theta^j\frac{(\theta t)^j}{j!}\binom{k-1}{j-1},
\end{equation}
where $\theta$ is the extremal index.

We now invoke a very recent theory developed by \cite{HHVV}, which allows us to put the statistics of the number of visits in a wider context and to make explicit the connection with the EI, which is the main object of this note.

We first replace the ball $B(z, r_n)$ with a sequence of sets $B_n(\Gamma)$ whose intersection over $n\in \mathbb{N}$ gives the zero measure set $\Gamma$. For instance in section \ref{SDEI} such sets are the neighborhoods $\Delta^k_n$. We keep in mind this example for the next application. We then define the {\em cluster size distribution} $\pi_l$ as
\begin{equation}\label{CSD}
\pi_l:=\lim_{n \rightarrow \infty}\lim_{K\rightarrow \infty}\frac{\mu(\sum_{i=0}^{2K}{\bf 1}_{B_n(\Gamma)}\circ T^i=l)}{\mu(\sum_{i=0}^{2K}{\bf 1}_{B_n(\Gamma)}\circ T^i \ge1)}.
\end{equation}
Under a few  smoothness assumptions for the map $T$, which are verified for the majority of what we called chaotic systems in this note, it is possible to prove (\cite{HHVV}, Th. 1) that
\begin{equation}\label{TTHH}
\mu\left(\sum_{l=1}^{\lfloor{\frac{t}{\mu(B_n(\Gamma))}}\rfloor}{\bf1}_{B_n(\Gamma)}(T^l(x))=k\right)\rightarrow \nu(\{k\}),
\end{equation}
as $n\rightarrow \infty$, where $\nu$ is the compound Poisson distribution for the parameter $t\pi_l$. This distribution can be recovered by its generating function $g_{\nu}(z)=\exp{\int_0^{\infty}(z^x-1)dH(x)}$, where $H$ is the counting measure on $\mathbb{N}$ defined as $H=\sum_lt\pi_l\delta_l$, with $\delta_l$ the point mass at $l$. The Poly\`a-Aeppli distribution is a particular case of compound Poisson distribution with the $\pi_l$ having a geometric distribution $\pi_l=\theta (1-\theta)^l$. The interesting feature of the theory in \cite{HHVV} is that the $\pi_l$ can be expressed in terms of more accessible quantities, in particular of the {\em distribution function $\hat{\alpha}_l$ of the $(l-1)$ return  into the set $B_n(\Gamma)$ for $n$ tending to infinity}, namely $\pi_l= \frac{\hat{\alpha}_l-2\hat{\alpha}_{l+1}+\hat{\alpha}_{l+2}}{\hat{\alpha}_1-\hat{\alpha}_2},$ where
\begin{equation}\label{alfa}
\hat\alpha_\ell=\lim_{K\to\infty}\lim_{n\to\infty}\mu_{B_n(\Gamma)}(\tau_{B_n(\Gamma)}^{\ell-1}\le K),
\end{equation}
being $\tau_{B_n(\Gamma)}^\ell(x)=\tau_{B_n(\Gamma)}^{\ell-1}(x)+\tau_{B_n(\Gamma)}(T^{\tau_U^{\ell-1}}(x)),$
 with $\tau_{B_n(\Gamma)}^1=\tau_{B_n(\Gamma)}=\min\{j\ge1: T^j(x)\in B_n(\Gamma), x\in B_n(\Gamma)\},$ the first return time.
 Moreover if $\theta$ denotes the extremal index (as defined in section \ref{SDEI} for general target sets), we have the identity:
\begin{equation}\label{REI}
\sum_{k=1}^{\infty} k \pi_k= \frac{1}{\theta}.
\end{equation}
This result is a generalization of Eq. (\ref{TH2}) stated in section 2 and which targeted neighborhoods of periodic points.

We now apply Eq. (\ref{REI}) to the dynamical extremal index introduced in section \ref{SDEI} in the case of direct product of two maps. For the class of maps verifying  Proposition \ref{EEII} we have the general result for the $\hat{\alpha}_l$ \cite{HHVV}:

$$\hat{\alpha}_{l+1}
 =\frac1{ \int_Ih^2(x) \,dx}\int_I \frac{h^{2}(x)}{|DT^l(x)|}\,dx.$$

This formula was also previously obtained by \cite{CCCC}.
In order to make rigorous computations, we consider a Markov map of the interval $T$ for which the density $h$ is piecewise constant (see for instance \cite{GB}):
\begin{align*}
T(x)&=3x\,\,,\quad x\in I_1:=[0,1/3)\\
&=5/3-2x\,\,,\quad x\in I_2:=[1/3,2/3)\\
&=-2+3x\,\,,\quad x \in I_3:=[2/3,1).
\end{align*}
The density reads

\begin{align*}
h(x)&=3/5=:h_1\,\,,\quad x\in I_1\\
&=6/5=:h_2\,\,,\quad x\in I_2\\
&=6/5=:h_3\,\,,\quad x \in I_3.
\end{align*}

\begin{figure}
\begin{tikzpicture}[xscale=6,yscale=5]

     \draw node [below] {$0$} (0,0) -- (1,0);
      \draw (0,0) -- (0,1)  node[left] {$1$};
      \draw[domain=0:1/3,thick,variable=\x,red] plot ({\x},{3*\x});
        \draw[domain=1/3:2/3,thick,variable=\x,red] plot ({\x},{5/3-2*\x});
          \draw[domain=2/3:1,thick,variable=\x,red] plot ({\x},{-2+3*\x});
          \draw [dashed] (1/3,0) node [below] {$1/3$} -- (1/3,1);
       \draw [dashed] (2/3,0) node [below] {$2/3$} -- (2/3,1/3);
        \draw (1,0) node [below] {$1$} -- (1,1);
        \draw  (0,1)  -- (1,1);
         \draw [dashed] (0,1/3) node [left] {$1/3$} -- (2/3,1/3);
    \end{tikzpicture}
\hspace{1cm}
\begin{tikzpicture}[xscale=6,yscale=3.2]

     \draw  node [below] {$0$} (0,0) -- (1,0);
      \draw (0,0) -- (0,1.5);
      \draw[domain=0:1/3,ultra thick,variable=\x,red] plot ({\x},{3/5});
        \draw[domain=1/3:2/3,ultra thick,variable=\x,red] plot ({\x},{6/5});
          \draw[domain=2/3:1,ultra thick,variable=\x,red] plot ({\x},{6/5});
           \draw [dashed] (1/3,0) node [below] {$1/3$} -- (1/3,6/5);
       \draw [dashed] (2/3,0) node [below] {$2/3$} -- (2/3,6/5);
        \draw  (1,0) node [below] {$1$} -- (1,1.5);

        \draw [dashed] (0,3/5) node [left] {$3/5$};
        \draw (0,1.5) -- (1,1.5);
        \draw [dashed] (0,6/5) node [left] {$6/5$} -- (1/3,6/5);
    \end{tikzpicture}
\caption{Graphical representation of the Markovian map (left) and its associated invariant density (right).}
\end{figure}
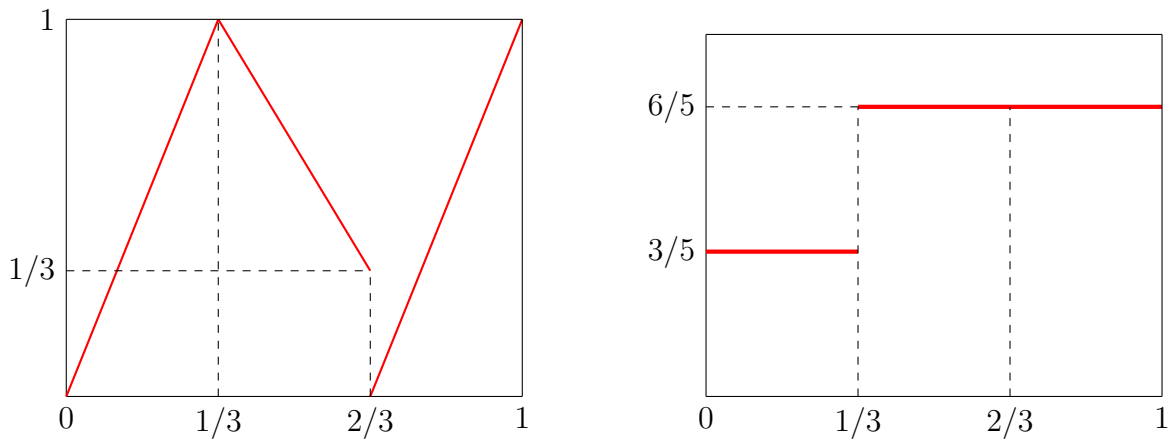

The extremal index is given by the integral in Eq. (\ref{EEII}) which has already been computed in \cite{CFMVY} and gives $\theta \approx 0.5926$. The quantities $\hat{\alpha}_l$ can also be computed explicitly.  We defer to \cite{GGHHVV} for the details, and obtain:
$$\hat\alpha_{k+1}=\frac{ \left[(3 \sqrt{145} - 37) (17 - \sqrt{145})^k + (37 + 3 \sqrt{145}) (17 + \sqrt{145})^k\right]}{ 2\times72^k\sqrt{145}}.$$
 The interesting point is that the $\pi_k$ do not follow a geometric distribution. For the statistics of the number of visits, we get a compound Poisson distribution which is not Poly\`a-Aeppli. To illustrate that, we show the distributions $\mu(N_n(t))=k$ for $t=50$ and different values of $k$ for the real process associated to the markov map $T$, and the distribution obtained by supposing a Poly\`a-Aeppli with parameter equal to $0.5926$. In this context, the sets $B_n(\Gamma)$ are  the tubular neighborhoods of the diagonal $\Delta_n^2$ defined in Eq. (\ref{deltak}). For our numerical simulation, we fix a small set $\Delta_n^2$ by taking $u_n$ in Eq. (\ref{deltak}) equal to the $0.99-$quantile of the distribution of the observable computed along a pre-runned trajectory of $10^4$ points. We obtain that for a small enough neighborhood of the diagonal, $\mu_2(\Delta_n^2) \approx (h_1^2+h_2^2+h_3^2)\sqrt{2}\exp{(-u_n)}$.

\begin{figure}[h!]
        \centering
        \includegraphics[height=2.in]{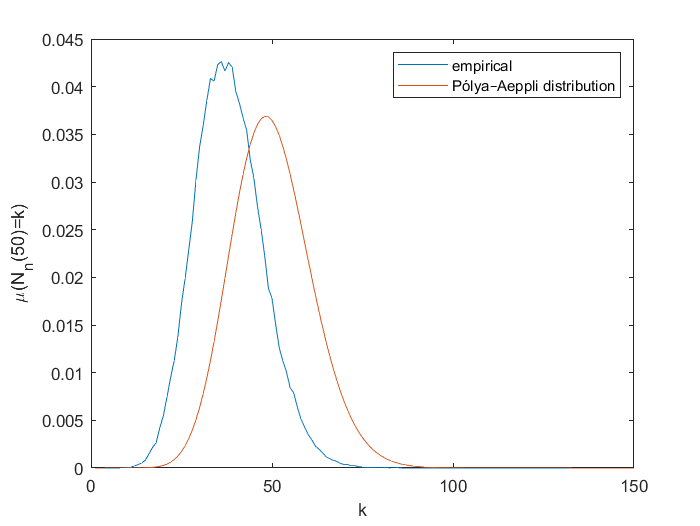}
        \caption{Comparison between the empirical distribution $\mu(N_n(50)=k)$ and the Poly\`a-Aeppli distribution of parameters $t=50$ and $\theta=0.5926$. The empirical distribution is obtained by considering $10^5$ runs for which is recorded the number of entrances in the set $\Delta_n^2$ ,as defined in the text, after $t/\mu_2(\Delta_n^2)$ iterations of the map.}
        \label{deiber}
    \end{figure}

\subsection{The random case}
There are only few results on the statistics of the number of visits in the case of random systems. In the forthcoming paper \cite{FFMV}, we show that the quenched Lasota-Yorke example studied in section \ref{CIR} gives a Poisson distribution for the visits around almost all target point $z$. A Poly\`a-Aeppli distribution is seen to hold for some particular random subshits.

We are not aware of a rigorous treatment  of these point processes for stationary (annealed) random systems. We believe that the theory developed in \cite{HHVV} could be applied as well to those systems. In particular we expect to get the extremal index as the inverse of the average cluster size. To corroborate this claim, we computed the statistics of the number of visits for two examples: the case in section \ref{DN} for the composition of two maps with $b=1/2$, for which there are infinitely many non zero $q_k$, and the moving target example in section \ref{MMTT}. In both cases we found that the real distribution coincide with Poly\`a-Aeppli with parameter given by $t=50$ and the extremal index of respectively is $2/3$ and $1-\frac32 (\frac14)^2-6(\frac14)^3(\frac12)^2 \approx 0.8828$. These results are stable against different values of $t$.

\begin{figure}[h!]
    \centering
    \begin{subfigure}[t]{0.5\textwidth}
        \centering
        \includegraphics[height=2.in]{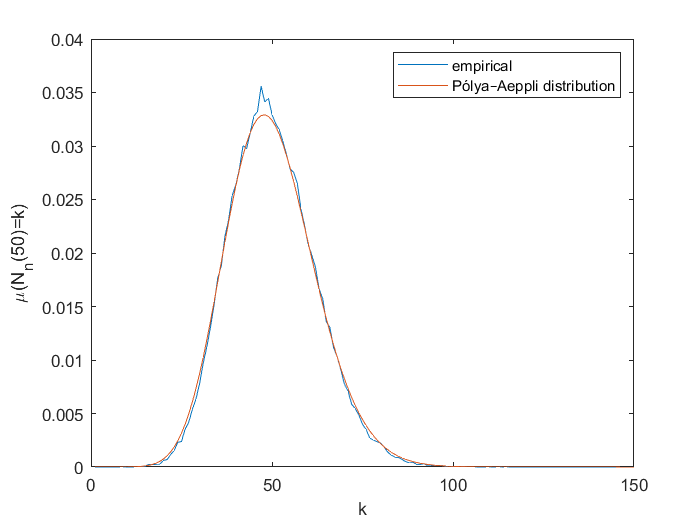}
        \caption{}
    \end{subfigure}%
    ~
    \begin{subfigure}[t]{0.5\textwidth}
        \centering
     \includegraphics[height=2in]{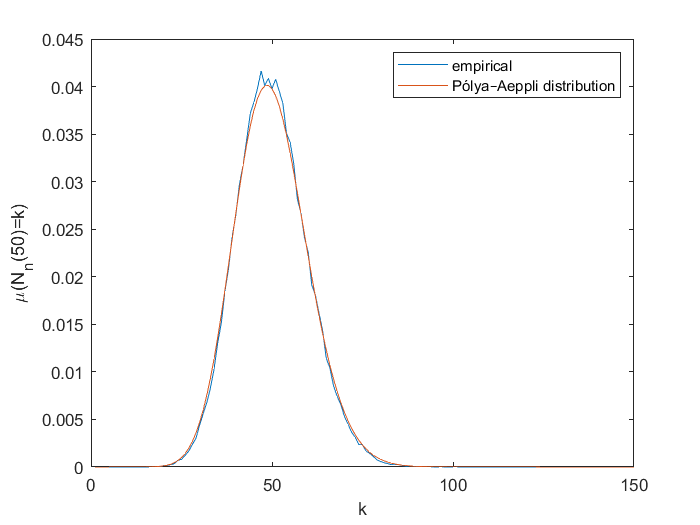}
         \caption{}
    \end{subfigure}
   \caption{For the composition of two maps with $b=1/2$ (left) and the moving target example (right), comparison between the empirical distribution $\mu(N_n(50)=k)$ and the Poly\`a-Aeppli distribution of parameters $t=50$ and respectively $\theta=2/3$ and $\theta=0.8828$. The empirical distribution is obtained as for the Markov maps, by considering $10^5$ runs. For both examples, we considered target balls of radius $r=\exp{(-u)}$, where $u$ is the 0.99-quantile of the empirical distribution of the observable computed along a pre-runned trajectory of size $10^4$.}
   \label{thetaq}
\end{figure}

\subsection{The rotational case}\label{roto}
We now compute the statistics of the number of visits for the process described in section \ref{ROT}. The statistics are computed in the same way as for the preceding examples and the theoretical value of 1 for the extremal index suggests that the number of visits is purely Poissonian. This is indeed what is observed in figure \ref{poisson}. Again the results are stable against different values of $t$.

\begin{figure}[h!]
        \centering
        \includegraphics[height=2.in]{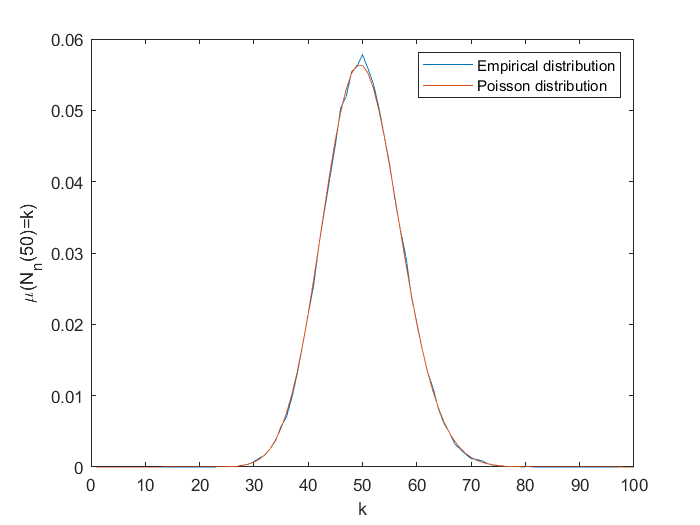}
        \caption{Comparison between the empirical distribution $\mu(N_n(50)=k)$ and the Poisson distribution of parameter $t=50$, computed with a trajectory of length $5.10^7$ of the rotational system described in section \ref{ROT}. We used the same parameters and procedure as the one in figure \ref{thetaq}.}
        \label{poisson}
    \end{figure}

\subsection{Climate data}

We study now the statistics of the number of visit for the atmospheric data presented in section \ref{est}. The target set considered in the analysis is a ball  centered at a point $z$ that corresponds to the mean sea-level pressure recorded on Jan 5th 1948. The radius $r=e^{-\tau}$, is fixed using as $\tau$  the $0.95-$quantile of the observable distribution. We considered the number of visits in an interval of length 365 days, that corresponds to the annual cycle\footnote{This time scale is relevant for the underlying dynamics of the atmospheric circulation.} so that the parameter $t$ in Eq. (\ref{PP1}) is taken equal to
$t=\lfloor365\mu(B(z,r))\rfloor$. For the target set considered, we get $t=18$, that is close to the mean value that we find for the number of visits of $18,13$. In figure \ref{visitclimate}, we compare the obtained empirical distribution with the Poly\`a-Aeppli distributions of parameters $t$ and the $EI$ obtained using different estimates: S\"uveges, $\hat\theta_5$ and $\hat\theta_{10}$. The distribution using the S\"uveges estimate seems to fit the empirical data better than the other two. In fact if we increase the order of the estimate $\hat\theta_m$, the value it attributes to the EI decreases and the associated Poly\`a-Aeppli distribution flattens. We believe that this is due to finite effects; indeed we observe that if the fixed threshold $u$ in formula \ref{MC} is too small, the numerator approaches 0 as $m$ increases, yielding bad estimates for the EI. Qualitatively similar results are found for different atmospheric states and different quantiles considered (although taking higher quantiles gives too few data for the statistics), suggesting that a Poly\`a-Aeppli distribution for the number of visits holds for this type of high dimensional observational data. This analysis, applied systematically to a state of a complex system, can provide important information on the recurrence properties of the target state, namely how many visit to expect in a given time interval. In this sense, the extremal index only provides an averaged information. The drawback is in the number of data required to provide a reliable estimates of the visits distribution (see Figure \ref{visitclimate}). This problem could be however overcome by averaging the distribution of visits for class of events (several states $z$) instead of considering just one state.

\begin{figure}[h!]
        \centering
        \includegraphics[height=2.in]{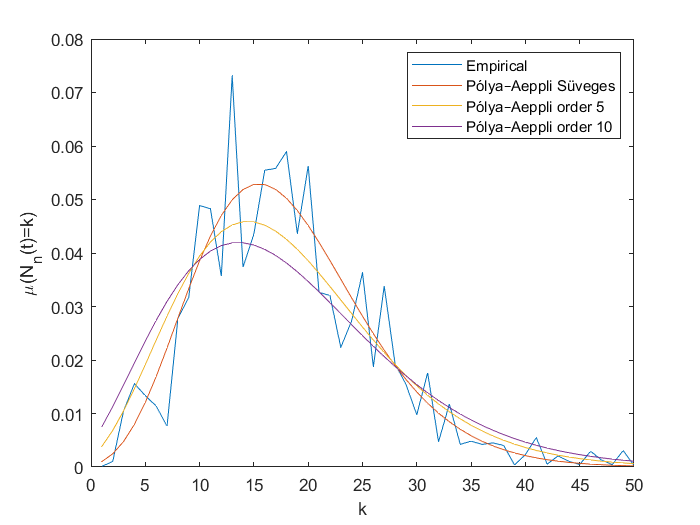}
        \caption{Comparison between the empirical distribution $\mu(N_n(t)=k)$ and the Poly\`a-Aeppli distributions of parameters $t=18$ and the EI found with different estimates. The empirical distribution is obtained by considering 24396 intervals of length 365 and counting the number of visits inside the target ball described in the text.}
        \label{visitclimate}
    \end{figure}

\section{Conclusions}

In this paper we have reviewed the meaning of the extremal index in dynamical systems and provided numerical and theoretical  strategies for its estimation. We have detailed how, depending on the nature (stochastic or deterministic) of the underlying system the estimators may provide different insights on the dynamics of the system.
Our main findings can be summarized as follows:
\begin{itemize}
\item As soon as a deterministic system is perturbed with noise having a smooth distribution, the EI becomes equal to $1$.  With discrete (point masses) distributions, the EI could be $1$ for almost choices of the target and of the realisations, (see sections \ref{CIR}, \ref{ROT}, \ref{AD}), but it could also  be  less than $1$, (see sections  \ref{DN}, \ref{DND}).
    \item What we described in the previous item happens either with annealed and quenched  stochastic perturbations. In the annealed case with discrete probability distributions, the value of the EI could depend either on  the choice of the masses and on the closeness of the maps (see Remark \ref{IR}). This suggests that in the situations where the EI is generically $1$, if one  finds a lesser value, this could be the effect of a discrete perturbation.
        \item The situation where the dynamics is deterministic, but the target set is known with some probability with a continuous distribution function is interesting, as discussed in section \ref{CCNN}. This describes physical situations where the localisation and observation of the extreme event are perturbed. No matter of the intensity of those perturbations, the extremal index is everywhere equal to one.
        \item The statistics of the number of visits ({\em persistence}) in the shrinking target set gives rise to point processes of Poisson type. We have shown that the geometric nature of the target set produces different compound Poisson distributions. Moreover these compound Poisson distributions persist even when the system is randomly perturbed. Indeed the fact that the extremal index is the inverse of the average cluster size seems to be a robust property of chaotic systems in the deterministic and random settings.
    \item In presence of periodicity one expects asymptotic distributions of  Poly\`a-Aeppli type. Nevertheless for  target sets  composed by more general invariant sets, different compound Poisson distribution could emerge  up as we showed  above. This suggest that for the statistics of the number of visits around sets with a more complicated geometry, and this could be relevant in applications, the computation of the EI is probably not enough, and one should go to higher moments to get the real asymptotic distribution.
\end{itemize}

\section{appendix}

As mentioned in section \ref{nonstat}, results proving the existence of extreme value laws for non-stationary sequential systems are lacking. Nevertheless, the methods of estimation described in the paper can still be used to evaluate a quantity $\theta$ that would correspond to the extremal index in the eventuality that an extreme value law holds for this kind of systems, which by the way was the case in \cite{FFV}. We could also alternatively define the extremal index  by computing the  statistics of the number of visit and evaluating the expectation in Eq. (\ref{REI}). We choose this second option and  consider the motion given by the concatenation

$$f_{\bar\xi}^n(x)=f_{\xi_n} \circ \cdots \circ f_{\xi_1} (x),$$
where the probability law of the $\xi_i$ changes over time. In particular, we consider the 10 maps
$$f_i(x)=2x+b_i \mod\ 1,$$
where $b_i$ is the $i^{th}$ component of a vector $\bar b$ of size 10, with entries equally spaced between $0$ and $1/2$. We consider sequences of time intervals $[\tau k+1,(k+1)\tau]$, with $\tau=10$, for $k=0,1,2,\dots$ in which the weights associated to $\xi_i$ are equal to $p_i^k$. For every $\tau$ iterations, the weight associated to each $\xi_i$ changes randomly, with the only constraint that they sum to 1. $\hat\theta_5$ is computed considering a trajectory of $5.10^7$ points, and a threshold $u$ corresponding to the $0.995-$quantile of the observable distribution. Using the same trajectory, we computed the empirical distribution of the number of visits in a ball centered at the origin and of radius $r=e^{-u}$ in intervals of time of length $\lfloor 2tr \rfloor$, with $t=50$. This is indeed $t$ times the Lebesgue measure (which is the invariant measure associated to all the maps $f_i$) of a ball of radius $r$ centered in the origin. Choosing the same trajectory for the computation of $\theta$ and for the statistics of visits is of crucial importance here, because different probability laws imply large variations of $\theta$ depending on the trajectory considered. In figure \ref{sequential}, we observe again a perfect agreement between the empirical distribution of the number of visits and the Poly\`a-Aeppli distribution of parameters $t=50$ and $\hat\theta_5$, which is equal to 0.92 for the trajectory presented in the figure \ref{sequential}. Although different trajectories give variations for $\theta$, this agreement is stable against different trajectories and different values of $t$ and $\tau$, suggesting that a geometric Poisson distribution for the number of visits given by $\theta$ is a universal feature, even for non-stationary scenarii.

\begin{figure}[h!]
        \centering
        \includegraphics[height=2.in]{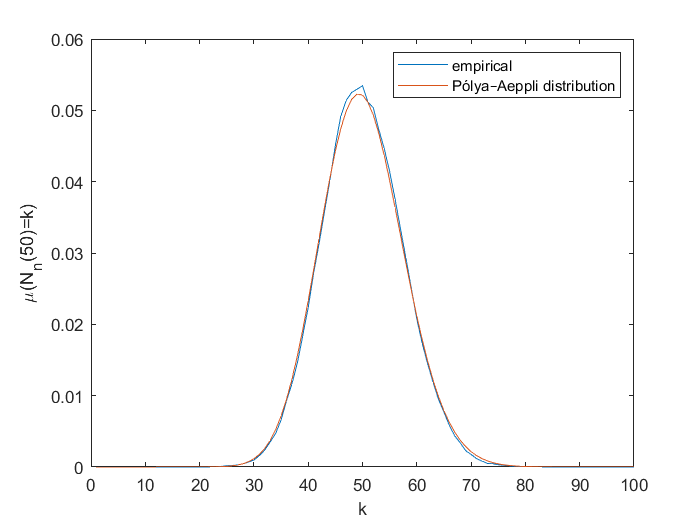}
        \caption{Comparison between the empirical distribution $\mu(N_n(50)=k)$ and the Poly\`a-Aeppli distribution of parameters $t=50$ and  $\hat\theta_5\approx 0.92$, computed with a trajectory of length $5.10^7$ of the sequential system described in the text. The procedure used to compute the empirical distribution is described in the text.}
        \label{sequential}
    \end{figure}

\section{Acknowledgement}
We would like to thank Jorge Freitas, Michele Gianfelice, Nicolai Haydn and Giorgio Mantica for several  discussions related to different parts of this work. PY was supported by ERC grant No. 338965-A2C2.



\end{document}